\documentclass[11pt,reqno]{amsart}
\usepackage{amsmath,amscd,amssymb,amsfonts,amsthm,courier,relsize,bm}
\usepackage{hyperref,enumerate,mathrsfs,mathtools,slashed}
\usepackage{scalerel}
\usepackage[all]{xy}

\usepackage{xcolor}  
\hypersetup{
    colorlinks,
    linkcolor={red!50!black},
    citecolor={blue!70!black},
    urlcolor={blue!80!black}
}

\newtheorem{theorem}{Theorem}[section]
\newtheorem{corollary}[theorem]{Corollary}
\newtheorem{proposition}[theorem]{Proposition}

\newtheorem{lemma}[theorem]{Lemma}
\theoremstyle{definition}    
\newtheorem{definition}[theorem]{Definition}
\theoremstyle{remark}

\newtheorem{remark}[theorem]{Remark}

\newtheorem{example}[theorem]{Example}

\newcommand{\pair}[2]{\langle #1, #2 \rangle}
\newcommand{\ignore}[1]{}
\newcommand{\matr}[4]{\left(\begin{array}{cc}#1&#2\\#3&#4\end{array}\right)}
\newcommand{\ol}[1]{\overline{#1}}

\newcommand{\sul}[1]{\underline{#1\mkern-4mu}\mkern4mu } 

\renewcommand{\sf}[1]{\mathsf{#1}}
\newcommand{\wh}[1]{\widehat{#1}}
\newcommand{\scr}[1]{\mathscr{#1}}
\newcommand{\mf}[1]{\mathfrak{#1}}

\newcommand{\tn}[1]{\textnormal{#1}}
\renewcommand{\i}{{\mathrm{i}}}
\renewcommand{\c}{{\mathrm{c}}}

\def\Dirac{\ensuremath{\slashed{\D}}}

\def\d{\ensuremath{\mathrm{d}}}

\def\Ad{\ensuremath{\textnormal{Ad}}}
\def\ad{\ensuremath{\textnormal{ad}}}
\def\g{\ensuremath{\mathfrak{g}}}

\def\t{\ensuremath{\mathfrak{t}}}
\def\n{\ensuremath{\mathfrak{n}}}

\def\h{\ensuremath{\mathfrak{h}}}
\def\hvee{\ensuremath{\textnormal{h}^\vee}}

\def\A{\ensuremath{\mathcal{A}}}
\def\B{\ensuremath{\mathcal{B}}}
\def\C{\ensuremath{\mathcal{C}}}
\def\D{\ensuremath{\mathcal{D}}}
\def\E{\ensuremath{\mathcal{E}}}

\def\G{\ensuremath{\mathcal{G}}}
\def\H{\ensuremath{\mathcal{H}}}
\def\I{\ensuremath{\mathcal{I}}}
\def\J{\ensuremath{\mathcal{J}}}

\def\L{\ensuremath{\mathcal{L}}}
\def\M{\ensuremath{\mathcal{M}}}

\def\S{\ensuremath{\mathcal{S}}}

\def\U{\ensuremath{\mathcal{U}}}

\def\X{\ensuremath{\mathcal{X}}}

\def\Z{\ensuremath{\mathcal{Z}}}

\def\bC{\ensuremath{\mathbb{C}}}
\def\bR{\ensuremath{\mathbb{R}}}
\def\bZ{\ensuremath{\mathbb{Z}}}

\def\bD{\ensuremath{\mathbb{D}}}

\def\End{\ensuremath{\textnormal{End}}}
\def\Hom{\ensuremath{\textnormal{Hom}}}

\def\ker{\ensuremath{\textnormal{ker}}}
\def\supp{\ensuremath{\textnormal{supp}}}
\def\id{\ensuremath{\textnormal{id}}}
\def\Aut{\ensuremath{\textnormal{Aut}}}

\def\Tr{\ensuremath{\textnormal{Tr}}}

\def\pr{\ensuremath{\textnormal{pr}}}
\def\dim{\ensuremath{\textnormal{dim}}}
\def\Sym{\ensuremath{\textnormal{Sym}}}
\def\dom{\ensuremath{\textnormal{dom}}}

\def\op{\ensuremath{\textnormal{op}}}
\def\Cl{\ensuremath{\textnormal{Cl}}}

\def\index{\ensuremath{\textnormal{index}}}

\def\sK{\ensuremath{\textnormal{K}}}
\def\aff{\ensuremath{\textnormal{aff}}}

\def\Ch{\ensuremath{\textnormal{Ch}}}
\def\Td{\ensuremath{\textnormal{Td}}}
\def\Eul{\ensuremath{\textnormal{Eul}}}

\def\Ahat{\ensuremath{\widehat{\textnormal{A}}}}

\def\hol{\ensuremath{\tn{hol}}}
\def\hotimes{\ensuremath{\wh{\otimes}}}

\def\ft{\ensuremath{[\![t]\!]}}
\def\aff{\ensuremath{\tn{aff}}}
\def\reg{\ensuremath{\tn{reg}}}

\usepackage[outline]{contour}

\def\DiracE{\ensuremath{\Dirac{}^{\E}}}
\def\DiracEs{\ensuremath{\Dirac{}^{\E,2}}}
\def\DiracQ{\ensuremath{\Dirac{}^Q}}
\def\DiracQs{\ensuremath{\Dirac{}^{Q,2}}}

\def\bkappa{\ensuremath{\bar{\kappa}}}

\title{Index formula for Hamiltonian loop group spaces}
\author{Yiannis Loizides}
\address{Department of Mathematics, George Mason University, 
Fairfax, Virginia, U.S.}
\email{yloizide@gmu.edu}

\begin{document}
\sloppy
\maketitle

\begin{abstract}
We study K-theory classes of Hamiltonian loop group spaces represented by admissible Fredholm complexes. We prove various equivariant index formulae in this context. In a sequel to this article we show that, when specialized to a family of non-local elliptic boundary value problems over the moduli space of framed flat connections on a surface, one obtains a gauge theory analogue of the Teleman-Woodward index formula.
\end{abstract}

\section{Introduction}
Let $G$ be a compact connected simply connected Lie group with Lie algebra $\g$. The loop group $LG$ of $G$ is the set of maps (smooth or perhaps of a fixed sufficiently large Sobolev class) from the circle into $G$, with the group operation given by pointwise multiplication. The geometry and representation theory of $LG$ have been studied extensively. There is a successful theory of projective positive energy representations of $LG$, closely related to representations of affine Kac-Moody Lie algebras. These representations can be constructed using the orbit method as spaces of holomorphic sections of line bundles over coadjoint orbits of $LG$, cf. \cite{PressleySegal}.

An important theme in symplectic geometry is the generalization from coadjoint orbits to arbitrary symplectic manifolds equipped with a Hamiltonian action. Systematic study of Hamiltonian $LG$-spaces $\M$ began with work of Meinrenken and Woodward about 20 years ago, cf. \cite{MWVerlindeFactorization}. Many of the fundamental results on finite dimensional Hamiltonian spaces now have analogues for Hamiltonian $LG$-spaces with proper moment map, including for example, the convexity theorem, Duistermaat-Heckman theory, Kirwan surjectivity, abelian and non-abelian localization, the quantization-commutes-with-reduction theorem, and the classification of multiplicity-free spaces. See for example \cite{LecturesGroupValued} for an overview of some of these developments.

There is an important collection of examples coming from 2D gauge theory. Let $\Sigma$ be a compact connected surface of genus $\rm{g}$ with $b\ge 1$ boundary components, and let $\sul{G}=G^b$. Atiyah and Bott \cite{AtiyahBottYangMills} observed that the space of $G$-connections $\A_\Sigma$ on $\Sigma$ carries a symplectic form and the action of the gauge group $\G_\Sigma$ is Hamiltonian. The symplectic quotient at $0$ is the (usually singular) moduli space of flat connections on the closed surface obtained by capping off the boundary of $\Sigma$. Reduction at levels other than $0$ produces more general moduli spaces of flat connections on $\Sigma$ with holonomy around the boundary components lying in prescribed conjugacy classes of $G$. Following Donaldson \cite{donaldson1992boundary}, one can instead reduce in stages, taking the symplectic quotient by the subgroup $\G_{\Sigma,\partial \Sigma}$ of gauge transformations that are trivial along the boundary. The result is a Hamiltonian $L\sul{G}$-space $\M_\Sigma=\A_\Sigma^\flat/\G_{\Sigma,\partial \Sigma}$, the moduli space of flat connections on $\Sigma$ with framing along the boundary. The moment map is induced by restriction of the connection to the boundary of $\Sigma$. Further symplectic reduction with respect to the $L\sul{G}$ action then recovers the moduli spaces of flat connections on $\Sigma$ with holonomy around the boundary components lying in prescribed conjugacy classes of $G$.

When specialized to the moduli space examples, the general theorems for Hamiltonian loop group spaces yield important results on moduli spaces of flat connections. For example, the general Duistermaat-Heckman theory for Hamiltonian loop group spaces leads to a proof of Witten's formulas for intersection pairings on moduli spaces of flat connections \cite{Meinrenken2005}. The general quantization-commutes-with-reduction theorem for Hamiltonian loop group spaces leads to a proof of the Verlinde formula for the Euler characteristic of a line bundle on moduli spaces of flat connections \cite{AMWVerlinde}. This article and its sequel are a further addition to this theory: the index formula for Hamiltonian loop group spaces that we prove in this article will be shown, in the sequel, to lead to a proof of the Teleman-Woodward formulae for K-theoretic intersection pairings on moduli spaces of flat connections.

The main object of study are what we shall call \emph{admissible} classes in the $G$-equivariant (resp. $T$-equivariant, for $T\subset G$ a maximal torus) K-theory $K_G(\M)$ (resp. $K_T(\M)$) of a Hamiltonian loop group space $\M$. Here $G$ acts on the $LG$-space $\M$ via the embedding of $G$ in $LG$ as the subgroup of constant loops. A K-class is admissible if it is represented by a Fredholm complex satisfying an equivariant bounded geometry condition; the precise definition is given in the body of the article. Such classes form an $R(G)$-subalgebra $K^\ad_G(\M)\subset K_G(M)$ (resp. $R(T)$-subalgebra $K^\ad_T(\M)\subset K_T(M)$). K-classes represented by finite-rank $LG$-equivariant vector bundles are admissible, but the definition is considerably more general. In a sequel to this article, we introduce a large collection of examples in case $\M=\M_\Sigma$, represented by families of elliptic boundary-value problems on $\Sigma$, which cannot be represented by $LG$-equivariant vector bundles. We refer to these as Atiyah-Bott K-theory classes, since they are a generalization of K-theory classes described by Atiyah and Bott \cite{AtiyahBottYangMills} for the moduli space associated to a closed surface.

Let $R(T)=\bZ[\Lambda]$ be the representation ring of a maximal torus $T$. Let $R^{-\infty}(T)=\bZ^\Lambda$ be the $R(T)$-module of formal infinite Fourier series. To extract invariants from admissible K-theory classes, we construct a map of $R(T)$-modules
\begin{equation} 
\label{e:indexmorphism}
\index_T \colon K^{\ad}_T(\M)\rightarrow R^{-\infty}(T) 
\end{equation}
out of the $R(T)$-subalgebra $K^{\ad}_T(\M)\subset K_T(\M)$ of $T$-equivariant admissible classes. The $\index_T$ of a class in the image of the natural map $K^\ad_G(\M)\rightarrow K^\ad_T(\M)$ has an additional Weyl group antisymmetry property, which allows one to formally induct up to a similar map
\[ \index_G\colon K^\ad_G(\M)\rightarrow R^{-\infty}(G).\] 
Conceptually \eqref{e:indexmorphism} pulls K-theory classes back to a finite-dimensional, non-compact, K-oriented (i.e. Spin$_c$) submanifold $X$ of $\M$, and then pushes forward to a point. Pushforwards in equivariant K-theory are implemented at the level of cycles by coupling vector bundles, or more general Fredholm families, to the Spin$_c$ Dirac operator, and then taking the equivariant index. The admissibility condition is used to show that the $L^2$-index on the non-compact manifold has finite $T$-multiplicities. The submanifold $X$, which was first described in \cite{LMSspinor}, plays a similar role to the extended moduli spaces used by Jeffrey and Kirwan \cite{JeffreyKirwanModuli}. 

The importance of the index map \eqref{e:indexmorphism} comes from the fact that it may be used to compute K-theoretic intersection pairings on finite dimensional symplectic quotients of $\M$ (for $\M=\M_\Sigma$, these quotients are moduli spaces of flat connections on $\Sigma$ with boundary holonomies in prescribed conjugacy classes). The relationship between the index map and K-theoretic intersection pairings stems from the non-abelian localization formula, and works as follows.

A Spin$_c$ structure can be twisted by tensor powers of a line bundle. A natural choice in this context, when it exists, is the pull back to $X$ of a prequantum line bundle $L\rightarrow \M$. In this case one gets a sequence of maps $\index_G^k\colon K^{\ad}_G(\M)\rightarrow R^{-\infty}(G)$ given as before except twisting the Spin$_c$ structure by $L^k$. Suppose $\sf{E} \in K^{\ad}_G(\M)$. Fix a fundamental alcove $\mf{A}^*\subset \t^*_+$, let $\xi \in \mf{A}^*$ be a regular value of the moment map and such that $k\xi \in \Lambda_+$ for some $k \in \bZ_{>0}$. Assume furthermore that $(LG)_\xi\subset G$ (this occurs if and only if $\xi$ is not in the back wall of the alcove). Then for $k\gg 0$, the multiplicity of the irreducible $G$-representation $V_{k\xi}$ with highest weight $k\xi$ in $\index_G^k(\sf{E})$ equals the index of the K-class $\sf{E}_{k\xi}$ induced by $\sf{E}\otimes L^k\otimes \bC_{-k\xi}$ on the finite-dimensional compact symplectic orbifold $\M_\xi=\mu_\M^{-1}(\xi)/(LG)_\xi$. This is a consequence of the non-abelian localization formula for the index to be discussed below. On the other hand, by the index formula for orbifolds, the index of $\sf{E}_{k\xi}$ is a quasi-polynomial function of $k$, hence is determined by its values for $k\gg 0$. In this sense the index homomorphism \eqref{e:indexmorphism} computes K-theoretic intersection pairings on $\M_\xi$. In case $\sf{E}=1$, the $[Q,R]=0$ theorem for Hamiltonian loop group spaces \cite{AMWVerlinde} (see also \cite{LSQuantLG, LSWittenDef, YiannisThesis}) says that equality between the index on $\M_\xi$ and the multiplicity of $V_{k\xi}$ holds for any $k>0$.

We prove three different formulae for the index \eqref{e:indexmorphism}. The first is a non-abelian localization formula (Theorem \ref{t:nonabelloc}). In its original form non-abelian localization refers to localizing the calculation of an integral (in cohomology) or of the index of an operator (in K-theory) on a Hamiltonian $G$-space to the critical set of the norm-square of the moment map $\|\mu\|^2$. There is a large literature on non-abelian localization in various forms, for example \cite{Kirwan, WittenNonAbelian, Paradan98, TianZhang, MaZhangVergneConj} amongst many others; in the more specific context of Hamiltonian loop group spaces references include \cite{BottTolmanWeitsman, WoodwardHeatFlow, DHNormSquare, LSWittenDef, YiannisThesis}. We prove our non-abelian localization formula by adapting a technique of Bismut-Lebeau \cite[Chapter IX]{BismutLebeau} to analyse the resolvents of a 1-parameter family of operators in the limit as the parameter goes to infinity.

Next we prove a Kirillov-Berline-Vergne delocalized formula (Theorem \ref{t:deloc}) expressing the equivariant index as an integral of $T$-equivariant characteristic forms over the submanifold $X\subset \M$. An attractive feature of this formula is that the symmetries of the index under the action of the affine Weyl group are realized concretely as symmetries of the characteristic forms. Finally we prove an Atiyah-Bott-Segal-Singer fixed-point formula (Theorem \ref{t:abelianloc}), which involves only integration of characteristic forms over a finite-dimensional compact fixed point set. The result is a generalization of the fixed-point formula of \cite{AMWVerlinde}, which covered the important special case $\sf{E}=1$ relevant to the proof of the Verlinde formula.

The main results can be combined and summarized in two theorems. More detailed explanation of terminology and notation will be given in the main body of the article.
\begin{theorem}
There is a well-defined equivariant index map
\[ \index_T\colon K^\ad_T(\M)\rightarrow R^{-\infty}(T),\]
sending an admissible K-theory class $\sf{E}$ represented by the admissible cycle $(\E,\nabla^\E,Q)$ to the $T$-equivariant $L^2$-index of $\Dirac^Q$. The index satisfies a non-abelian localization formula
\begin{equation} 
\label{e:nonabellocintro}
\index_T(\sf{E})=\sum_{\beta \in \B}\index_T(\sigma_\beta\otimes \sf{E}_\beta), \qquad \B=\{\beta \in \t\mid X^\beta\cap \phi^{-1}(\beta)\ne \emptyset\}.
\end{equation}
\end{theorem}
\noindent Here $\Dirac^Q$ is the operator on $X\subset \M$ obtained by coupling $(\E,\nabla^\E,Q)$ to the Dirac-type operator $\Dirac$ on $X$. On the right hand side of \eqref{e:nonabellocintro}, the sum over $\B$ is infinite but locally finite when the summands are viewed as functions on the weight lattice $\Lambda$, and $\phi$ is a moment map. Moreover $\sigma_\beta\otimes \sf{E}_\beta$ is a $T$-transversely elliptic symbol with non-invertibility set contained in $X^\beta\cap \phi^{-1}(\beta)$, and its index is taken in the sense of Atiyah \cite{AtiyahTransEll}.

The second theorem applies under more restrictive hypotheses. In particular for the fixed-point formula, we assume existence of a formal series $\sf{E}_t \in \bC K^\ad_T(\M)\ft$ of admissible classes that changes by a multiplicative factor in $1+t\bC R(T)\ft\subset GL_1(\bC R(T)\ft)$ under pullback by elements of the subgroup $\Pi=\Hom(S^1,T)\subset LG$. The group $GL_1(\bC R(T)\ft)$ acts on $\bC K^\ad_T(\M)\ft$ via the $R(T)$-module action on $K^\ad_T(\M)\ft$. Twisted $\Pi$-equivariance is thus described by a group homomorphism 
\[ \Pi \rightarrow 1+t\bC R(T)\ft, \qquad \eta \mapsto \exp(tv_t(u)\cdot \eta), \] 
where $v_t \in \t_\bC\otimes \bC R(T)\ft$ is a formal series of vector fields on the complexified torus $T_\bC$. Twisted equivariance can be thought of as a formal deformation of strict $\Pi$-equivariance that one has at $t=0$. This condition is difficult to motivate here, but we will see it appear naturally in the study of the Atiyah-Bott classes for $\M=\M_\Sigma$ in the sequel to this article. It is closely related to a famous change of variables used by Witten \cite{Witten} (see also \cite{JeffreyKirwanModuli, Meinrenken2005}) for the calculation of cohomological intersection pairings on moduli spaces of flat connections, and also plays a key role in the article \cite{TelemanWoodward} of Teleman and Woodward.
\begin{theorem}
Let $\sf{E}_t\in \bC K^\ad_T(\M)\ft$ be a formal series of admissible classes that admits infinitesimally twisted $T\times \Pi$-equivariant Chern character forms. Then
\[ \index_T(\sf{E}_t)(u\exp(\xi))=\int_{X^u}\Ch^u(\sf{E}_t,\xi)\A\S^u(\sigma,\xi), \]
for $\xi \in \t$ sufficiently small. Let $\Phi_t(u)=u\exp(t\ell^{-1}v_t(u))$ where $v_t \in \t_\bC\otimes \bC R(T)\ft$ is the infinitesimal twist. Let $g_t \in \J^\infty_g T_\bC$ be the solution of the fixed-point equation $\Phi_t(g_t)=g$. Then
\[ \index_{T}(\sf{E}_t)(u)=\frac{1}{|T_\ell|}\sum_{g \in T_\ell}\frac{\delta_{g_t}(u)}{\det(1+t\ell^{-1}\d v_t(g_t))}\int_{X^g/\Pi} \Ch^{g_t}(\sf{E}_t)\A\S^{g_t}(\sigma). \]
\end{theorem}
\noindent Here $\A\S^u(\sigma,\xi)$ denotes the equivariant Atiyah-Singer integrand for an elliptic symbol $\sigma$ (it may also be written entirely in terms of standard characteristic classes). The integral in the Kirillov-Berline-Vergne formula converges in the sense of distributions (in the variable $\xi \in \t$). The variable $t$ is a formal parameter, and $g_t$ is an $\infty$-jet of a curve at $g \in T_\bC$, or equivalently, a $\bC\ft$-point in $T_\bC$. $\ell>0$ is the level of the spinor bundle, and $T_\ell=\ell^{-1}\Lambda/\Pi\subset T$ is a finite subgroup. There are similar statements for $\index_G\colon K^\ad_G(\M)\rightarrow R^{-\infty}(G)$ exhibiting the additional symmetry.

In a sequel to this article, we show that when the fixed-point formula is applied to the Atiyah-Bott classes on $\M=\M_\Sigma$, one obtains formulae precisely matching the algebro-geometric index formulae of Teleman and Woodward \cite{TelemanWoodward} on the moduli stack of holomorphic principal $G_\bC$ bundles on a closed Riemann surface. The latter are a remarkable generalization of the Verlinde formula first conjectured by Teleman \cite{TelemanDefVerlinde}. Our theorem and that of Teleman and Woodward apply to different contexts and do not admit a direct comparison, but both have the formulae for K-theoretic intersection pairings on quotients for $k\gg 0$ as a common consequence.

We expect that several of the methods we develop are applicable to the index theory of other families of boundary value problems in gauge theory and differential geometry, and so may be of broader interest.

The contents of the article are as follows. In Section \ref{s:background} we recall the definition of a Hamiltonian loop group space $\M$, and describe the K-oriented non-compact submanifold $X \subset \M$ that plays a fundamental role throughout the article. We also establish some special properties of the Kirwan vector field on $X$. In Section \ref{s:quasi} we introduce the subalgebra of admissible K-theory classes, and prove that such classes have an equivariant analytic index. We also prove the non-abelian localization formula for the index in this section. In Section \ref{s:cohomology} we discuss desirable analytic properties of differential form representatives of the Chern character. We then prove the Kirillov-Berline-Vergne and fixed-point formulae for the index.

\bigskip

\noindent \textbf{Acknowledgements.} I thank Michele Vergne for discussions and for introducing me to interesting work in theoretical physics \cite{EquivVerlinde} on the equivariant Verlinde formula several years ago. I thank Eckhard Meinrenken, Nigel Higson, Reyer Sjamaar, Chris Woodward, Maxim Braverman, and Yanli Song for helpful discussions.

\bigskip

\noindent \textbf{Notation.} We often deal with vector spaces/bundles that are $\bZ_2$-graded and in this context $[a,b]$ denotes the graded commutator of linear operators $a,b$. Graded tensor products are denoted $\hotimes$.

Let $G$ be a compact connected and simply connected Lie group with Lie algebra $\g$. Fix a maximal torus $T \subset G$ with Lie algebra $\t$, and a choice of positive chamber $\t_+\subset \t$. Let $W$ be the Weyl group, and let $\mf{R} \supset \mf{R}_+$ be the set of roots, positive roots respectively. Let $n_+=|\mf{R}_+|$ be the number of positive roots. The integer or co-root lattice is denoted $\Pi=\ker(\exp\colon \t \rightarrow T)\simeq\Hom(U(1),T)$, and the real weight lattice by $\Lambda=\Hom(\Pi,\bZ)\simeq\Hom(T,U(1))$. The lattice $\Pi$ in $\t$ determines a canonical normalization of Lebesgue measure, that we use to identify distributions with generalized functions. Usually $u,g$ denote elements of $T$ (or $T_\bC$), $\xi$ an element of $\t$ (or $\t_\bC$), $\eta$ an element of $\Pi$, and $\lambda$ an element of $\Lambda$. 

The representation ring of $T$ is $R(T)=\bZ[\Lambda]$, and $\bC R(T)=\bC\otimes R(T)$. The formal completion $R^{-\infty}(T)=\bZ^\Lambda$ is an $R(T)$-module. If $T$ acts on a Hilbert space $H$, the isotypical subspace of weight $\lambda \in \Lambda$ is denoted $H_{[\lambda]}$. If $a$ is a $T$-equivariant linear operator on $H$, then the induced operator on $H_{[\lambda]}$ is denoted $a_{[\lambda]}$.

For a simple Lie algebra, every invariant integral inner product is a positive integer multiple of the basic inner product, normalized such that the length of the short co-roots is $\sqrt{2}$. For $\mf{su}(n)$ the basic inner product is
\[ (\xi_1,\xi_2)\mapsto \xi_1\cdot\xi_2=-\frac{1}{4\pi^2}\Tr_{\bC^n}(\xi_1\xi_2).\] 
The Lie group $G=G_1\times \cdots \times G_m$ is a product of its simple factors and $\g=\g_1\oplus \cdots \oplus \g_m$. Using the basic inner product on each summand $\g_i$, we obtain an inner product `$\cdot$' on $\g$ that we use throughout the article to identify $\g$ with $\g^*$. An $m$-tuple $\ell=(\ell_1,...,\ell_m) \in \bZ_{>0}^m$ determines a map $\ell \colon \g \rightarrow \g$ given by multiplication by $\ell_i$ on the summand $\g_i$, and $\ell^{-1}$ denotes the inverse map. For $\ell \in \bZ^m_{>0}$, define a finite subgroup $T_\ell:=\ell^{-1}\Lambda/\Pi \subset \t/\Pi=T$.

\section{Background on Hamiltonian loop group spaces}\label{s:background}
In this section we review the definition and needed properties of Hamiltonian loop group spaces $\M$. In particular we recall the construction of a finite dimensional smooth Spin$_c$ submanifold $X \subset \M$ that plays a fundamental role throughout the rest of the paper, analogous to the role played by the `extended moduli spaces' of Jeffrey-Kirwan \cite{JeffreyKirwanModuli}. In the last subsection we prove some special properties of the Kirwan vector field on $X$ that will play a role in analytic considerations of later sections.

\subsection{Loop groups.}\label{s:lg}
For more detailed background on loop groups see \cite{PressleySegal}. Let $LG$ denote the space of maps $S^1=\bR/\bZ \rightarrow G$ of fixed Sobolev class $\varsigma>\frac{1}{2}$. This is a Banach Lie group with group operation given by pointwise multiplication and Lie algebra $L\g$ identified with $\g$-valued $0$-forms $\Omega^0_\varsigma(S^1,\g)$ of Sobolev class $\varsigma$.

The $U(1)$ central extensions of the loop group are classified by invariant symmetric bilinear forms on $\g$ taking integer values on $\Pi$. For a simple Lie algebra such a form is multiple of the basic inner product. Thus central extensions of $LG$ are in one-one correspondence with $m$-tuples $\ell \in \bZ^m$. The corresponding central extension $\wh{LG}^{(\ell)}$ with Lie algebra cocycle
\[ (\xi_1,\xi_2)\mapsto 2\pi \i \int_{S^1} \ell\d\xi_1\cdot \xi_2,\]
is called the \emph{level} $\ell$ central extension of $LG$. The group $G$ sits inside $LG$ as the subgroup of constant loops. The central extension $\wh{LG}^{(\ell)}$ trivializes over $G$, hence $G$ can be regarded as a subgroup of $\wh{LG}^{(\ell)}$. Another subgroup of $LG$ that plays a key role for us is the integral lattice $\Pi=\ker(\exp \colon \t \rightarrow T)$ where $\eta \in \Pi$ is identified with the 1-parameter subgroup $s \in S^1\mapsto \exp(-s\eta) \in T$. The restriction of the central extension $\wh{LG}^{(\ell)}$ to the subgroup $T\times \Pi$  is a semi-direct product $T\ltimes \wh{\Pi}^{(\ell)}$, where elements $u \in T$, $\wh{\eta}\in \wh{\Pi}^{(\ell)}$ obey the commutation relation
\begin{equation}
\label{e:commutation}
u\wh{\eta}u^{-1}\wh{\eta}^{-1}=u^{\ell\eta}=e^{2\pi \i \ell\xi\cdot \eta}, \qquad u=\exp(\xi).
\end{equation}

Let $L\g^*=\Omega^1_{\varsigma-1}(S^1,\g)$ denote the space of $\g$-valued $1$-forms (or connections) of Sobolev class $\varsigma-1$. Since $\varsigma-1>-\varsigma$, there is a non-degenerate pairing $L\g \times L\g^*\rightarrow \bR$ given in terms of the inner product $\cdot$ on $\g$ followed by integration over $S^1$. The loop group $LG$ acts smoothly and properly on $L\g^*$ by gauge transformations:
\begin{equation} 
\label{e:gauge}
a\mapsto g\cdot a=\Ad_g(a)-\d gg^{-1}, \qquad g \in LG, a \in L\g^*.
\end{equation}
The gauge action \eqref{e:gauge} coincides with the restriction of the coadjoint action of $\wh{LG}=\wh{LG}^{(1,1,...,1)}$ to the affine hyperplane $L\g^*\times \{1\}\subset \wh{L\g}^*$. Given $\xi \in L\g$ the corresponding vector field for the infinitesimal action on $L\g^*$ is $\xi_{L\g^*}(a)=\partial_u|_{u=0} \exp(-u\xi)\cdot a=\d_a \xi$ at $a \in L\g^*$, where $\d_a=\d+\ad_a$ is the covariant derivative defined by the connection $a$. The subgroup $\Omega G \subset LG$ consisting of loops beginning at $1 \in G$ acts on $L\g^*$ freely with quotient $L\g^*/\Omega G \simeq G$, the isomorphism being induced by the map
\[ \hol \colon L\g^* \rightarrow G \]
taking a connection to its holonomy.

\subsection{Hamiltonian loop group spaces.}\label{s:hlgsp}
\begin{definition}[cf. \cite{MWVerlindeFactorization}]
A \emph{proper Hamiltonian} $LG$-\emph{space} $(\M,\omega_\M,\mu_\M)$ is a Banach manifold $\M$ equipped with a smooth action of $LG$, an $LG$-invariant weakly symplectic form $\omega_\M$ and an $LG$-equivariant proper moment map:
\[ \mu_\M \colon \M \rightarrow L\g^*, \qquad \iota(\xi_\M)\omega_\M=-\d \pair{\mu_\M}{\xi}, \quad \xi \in L\g.\]
(We will often drop the adjective `proper' for brevity.)
\end{definition} 

Since $\Omega G$ acts freely and properly on $L\g^*$ and $\mu_\M$ is proper and equivariant, $\Omega G$ acts freely and properly on $\M$ with quotient $M=\M/\Omega G$ a smooth compact $G$-manifold. Let $\hol_\M \colon \M \rightarrow M$ be the quotient map. The moment map $\mu_\M$ descends to an equivariant `group-valued moment map' $\mu_M\colon M \rightarrow G$ such that $\mu_M\circ \hol_\M=\hol\circ \mu_\M$; see \cite{AlekseevMalkinMeinrenken} for much more from this perspective. Under our assumptions on $G$, it is known that $M$ is automatically even-dimensional.

\begin{remark}
In this article we will not use the symplectic form $\omega_\M$ directly, but only a Spin$_c$ structure derived from it (see Section \ref{s:spinor}). Thus for example the non-degeneracy (or even existence) of $\omega_\M$ could be relaxed, assuming instead the existence of a suitable Spin$_c$ structure.
\end{remark}

\begin{example}\label{ex:modspace}
Let $\Sigma$ be a compact connected oriented surface with $b\ge 1$ parametrized boundary components. Let $\sul{G}=G^b$. Let
\[ \A_\Sigma=\Omega^1_{\varsigma-\frac{1}{2}}(\Sigma,\g) \supset \A_\Sigma^\flat \]
be the space of connections, resp. the space of flat connections on the principal $G$-bundle $\Sigma \times G$, of Sobolev class $\varsigma-\frac{1}{2}$. Let $\G_\Sigma$ be the gauge group, that is, the set of maps $\Sigma \rightarrow G$ of Sobolev class $\varsigma+\frac{1}{2}$. With this choice of Sobolev class there is a well-defined continuous group homomorphism $\G_\Sigma\rightarrow L\sul{G}$ given by restriction (or trace) to the boundary. Since $G$ is simply connected, $\scr{R}$ is surjective, hence we obtain a short exact sequence
\[ 1 \rightarrow \G_{\Sigma,\partial \Sigma}\rightarrow \G_\Sigma \rightarrow L\sul{G} \rightarrow 1.\]
The affine space $\A_\Sigma$ carries the well-known Atiyah-Bott symplectic structure \cite{AtiyahBottYangMills}, and the action of $\G_{\Sigma,\partial \Sigma}$ is Hamiltonian for the moment map $A \mapsto \tn{curv}(A)$. The symplectic quotient
\[ \M_\Sigma=\A_\Sigma^\flat/\G_{\Sigma,\partial \Sigma} \]
is a proper Hamiltonian $L\sul{G}$-space, where the moment map is given by the pullback of the connection to the boundary. For further details see \cite{MWVerlindeFactorization, meinrenken1999cobordism}.
\end{example}

\subsection{Levels and spinor bundles.}\label{s:spinor}
\begin{definition}
Let $\M$ be a Hamiltonian $LG$-space. A vector bundle $E \rightarrow \M$ is said to be at \emph{level} $\ell \in \bZ^m$ if it carries an action of $\wh{LG}^{(\ell)}$ compatible with the action of $LG$ on $\M$.
\end{definition}

As mentioned in the introduction, Spin$_c$ structures will be important in this article, because they play a role in K-theory analogous to that played by orientations in cohomology. Spin$_c$ structures on an even rank Euclidean vector bundle $V$ are in one-one correspondence with \emph{spinor bundles}: $\bZ_2$-graded Hermitian vector bundles $\scr{S}=\scr{S}^+\oplus \scr{S}^-$ together with an isomorphism of bundles of $\bZ_2$-graded $^*$-algebras
\begin{equation} 
\label{e:cliffstructmap}
\c \colon \bC l(V)\rightarrow \End(\scr{S}),
\end{equation}
where $\bC l(V)$ is the bundle of complexified Clifford algebras of $V$ (with fibres isomorphic to complex square matrices of size $2^{\tn{rank}(V)/2}$), cf. \cite{LawsonMichelsohn}.

Fix a $G$-invariant Riemannian metric on the compact manifold $M=\M/\Omega G$. The pullback $\hol_\M^*TM$ is a finite rank $LG$-equivariant Euclidean vector bundle over $\M$, hence we may define the Clifford algebra bundle $\bC l(\hol_\M^*TM)$.
\begin{definition}
\label{d:spinormodule}
A \emph{level} $\ell$ \emph{spinor bundle} for $\hol_\M^*TM$ is a spinor bundle $\scr{S}$ for $\hol_\M^*TM$ such that the underlying complex vector bundle is at level $\ell$ and the structure map \eqref{e:cliffstructmap} is $LG$-equivariant.
\end{definition}
\begin{remark}
Note that if $\scr{S}$ is at level $\ell$ then $\scr{S}^*$ is at level $-\ell$ and $\End(\scr{S})$ is at level $\ell-\ell=0$, hence carries an $LG$-action.
\end{remark}
 
Let $\hvee=(\hvee_1,...,\hvee_m) \in \bZ^m_{>0}$ denote the $m$-tuple of dual Coxeter numbers of the simple Lie algebras $\g_i$ in the direct sum decomposition $\g=\g_1\oplus \cdots \oplus \g_m$.
\begin{theorem}[\cite{LMSspinor}]
\label{t:spinor}
Let $\M$ be a proper Hamiltonian $LG$-space with quotient $M=\M/\Omega G$. There is a canonical level $\hvee$ spinor bundle $\scr{S}_{\tn{can}}$ for $\hol_\M^*TM$.
\end{theorem}
The heuristic idea behind the construction in \cite{LMSspinor} is as follows. In finite dimensions, spinor modules satisfy $2$-out-of-$3$: given a short exact sequence of even rank Euclidean vector bundles $0\rightarrow V_1 \rightarrow V_2 \rightarrow V_3 \rightarrow 0$ and given spinor modules for any two of the $V_i$, one obtains a spinor module for the third. The manifold $\M$ is symplectic, hence has a spinor module coming from a choice of compatible almost complex structure. The kernel of the quotient map $T\hol_\M\colon T\M \rightarrow \hol_\M^*TM$ is trivial with fibres isomorphic to $\Omega \g\simeq L\g/\g$. The latter has a canonical complex structure with $+\i$-eigenspace given by the positive Fourier modes, and hence a canonical spinor module as well. Thus assuming the $2$-out-of-$3$ property still holds in this infinite dimensional context, one obtains a spinor module for $\hol_\M^*TM$. The $2$-out-of-$3$ property should not be expected to hold in a general infinite-dimensional context because of the non-uniqueness of irreducible representations of Clifford algebras in infinite dimensions; however it can be shown to apply in this case \cite{LMSspinor}. The dual Coxeter number appears because this is the level of the spin representation of the loop group.

Further important examples of spinor bundles come from tensoring with line bundles: given a level $k$ line bundle $L$ on $\M$, one obtains a new spinor module $\scr{S}=\scr{S}_{\tn{can}}\otimes L$ at level $\ell=k+\hvee$. Line bundles appear naturally in the context of geometric quantization.

\begin{remark} 
For $\M=\M_\Sigma$ the moduli space of flat connections of Example \ref{ex:modspace}, prequantum line bundles are discussed in \cite[Section 3.3]{MWVerlindeFactorization}, following \cite{Mickelsson,RamadasSingerWeitsman,Witten}: $\M_\Sigma$ has a canonical $\wh{L\sul{G}}$-equivariant prequantum line bundle, obtained by reduction of a prequantum line bundle for the Atiyah-Bott symplectic structure on $\A_\Sigma$.
\end{remark}

From now on we assume that the level $\ell$ of $\scr{S}$ satisfies
\[ \text{ \fbox{ $\ell_1 > 0,...,\ell_m>0$, } } \]
to be abbreviated $\ell>0$ below. In particular for $\scr{S}_{\tn{can}}\otimes L$ we require that $k>-\hvee$.

\begin{definition}
\label{d:spincmoment}
Let $\scr{S}\rightarrow \M$ be a level $\ell$ spinor bundle for $\hol_\M^*TM$. The \emph{anti-canonical line bundle} of $\scr{S}$ is the level $2\ell$ line bundle
\[ \scr{L}=\Hom_{\bC l(\hol_\M^*TM)}(\scr{S}^*,\scr{S}).\]
Choose an invariant Hermitian connection $\nabla^{\scr{L}}$ on $\scr{L}$. Let 
\[ \varpi=\frac{1}{2}c_1(\scr{L},\nabla^{\scr{L}}) \] 
be half the first Chern form associated to the connection. The \emph{Spin}$_c$ \emph{moment map} is the map $\phi_\M\colon \M \rightarrow L\g^*$ defined by
\[ 2\pi \i\pair{\phi_\M}{\xi}=\frac{1}{2}(\L^{\scr{L}}_\xi-\nabla^{\scr{L}}_{\xi_\M}), \qquad \xi \in L\g.\]
\end{definition}
The Spin$_c$ moment map $\phi_\M$ is $LG$-equivariant with respect to the level $\ell$ gauge action 
\begin{equation}
\label{e:levelellgauge}
a\mapsto \Ad_g(a)-\ell \d g g^{-1},
\end{equation} 
on $L\g^*$.

\subsection{Global transversals.}\label{s:transversal}
The Lie algebra $\t$ can be viewed as the subspace of $L\g^*=\Omega^1(S^1,\g)$ consisting of constant $\t$-valued 1-forms. Under the gauge action it is preserved by the subgroup $N(T)\ltimes \Pi \subset LG$ where $N(T)$ is the normalizer of $T$ in $G$. Since $T \subset N(T)$ acts trivially on $\t$, there is an induced action of $(N(T)\ltimes \Pi)/T=W\ltimes \Pi$, which one checks is the standard action of the affine Weyl group $W_{\tn{aff}}=W\ltimes \Pi$ on $\t$.

Recall from Section \ref{s:lg} that the based loop group $\Omega G$ acts freely on $L\g^*$. The subspace $\t$ is not transverse to the orbits of the $\Omega G$ action. However it is possible to `thicken' $\t$ to obtain a finite dimensional submanifold of $L\g^*$ with this property. Let $\t^\perp$ be the orthogonal complement to $\t$ inside $\g$. Let $N(T)\ltimes \Pi$ act on $\t$ as above, and on $\t^\perp$ via the adjoint action of $N(T)$ (with $\Pi$ acting trivially). Let $\ol{B}_\epsilon(\t^\perp)$ denote the closed ball in $\t^\perp$ of radius $\epsilon$.
\begin{proposition}[\cite{LMSspinor}, Section 6.4]
\label{p:trans}
For $\epsilon>0$ sufficiently small, the inclusion $\t \hookrightarrow L\g^*$ extends to an $N(T)\ltimes \Pi$-equivariant embedding
\[ \t \times \ol{B}_\epsilon(\t^\perp) \rightarrow L\g^* \]
with image $R \subset L\g^*$ transverse to the $\Omega G$ orbits and contained in the dense subspace of smooth loops. Each $\Omega G$ orbit intersects $R$ non-trivially in an orbit of $\Pi$.
\end{proposition}
Sketching the proof, first note that via a tubular neighborhood embedding, it is enough to find an $N(T)\ltimes \Pi$-invariant complementary subbundle to the $\Omega G$-orbit directions of the normal bundle $\nu(L\g^*,\t)=\t \times (L\g^*/\t)$ (hereafter referred to as a `splitting'). Under the quotient map $\hol \colon L\g^* \rightarrow G$, the subspace $\t$ maps to the maximal torus $T=\t/\Pi$. Since the normal bundle $\nu(G,T)\simeq T \times \t^\perp$, at any point $a \in \t$ we can find a subspace of $\nu_a(L\g^*,\t)$ mapped by the normal derivative $\nu_a(\hol)\colon \nu_a(L\g^*,\t)\rightarrow \nu_{\exp(a)}(G,T)$ isomorphically to $\nu_{\exp(a)}(G,T)\simeq \t^\perp$. These splittings can be chosen to depend smoothly on $a \in \t$ and, since $N(T)\ltimes \Pi$ acts properly on $\t$, they can be chosen $N(T)\ltimes \Pi$-equivariantly. One can arrange that $R$ is contained in the space of smooth loops as well. The result is canonical up to homotopy, the only choices involved being a splitting (preserved under convex combinations) and a `thickness' $\epsilon$.

\begin{example}
We make the splitting of $\nu(L\g^*,\t)$ more explicit in the case $\g=\mf{su}(2)$. In this case $\t \subset L\g^*$ is $1$-dimensional. The tangent space to the $LG$ orbit through the point $a \in \t \subset L\g^*$ is the range of the covariant derivative $\d_a \colon L\g \rightarrow L\g^*$, and is closed of finite codimension since $\d_a$ is Fredholm. The orthogonal complement of the range is the kernel of the $L^2$ adjoint $\d_a^*=-\d_a$. This subspace is
\begin{equation} 
\label{e:kerda}
\ker(\d_a^*)=\{\xi(s)=\Ad_{\exp(-sa)}(\xi_0)|\xi_0\in \g_{\exp(a)}\} \supset \t
\end{equation}
where $\g_{\exp(a)}$ is the Lie algebra of the centralizer $G_{\exp(a)}$ in $G$ of $\exp(a)\in T$. Let $\Gamma=\exp^{-1}(\{1,-1\})=\frac{1}{2}\Pi$ be the lattice of elements in $\t$ that exponentiate to an element of the center of $G$. When $a=c \in \Gamma$ then $\ker(\d_c^*)\simeq \g$ and $\{c\}\times \ker(\d_c^*)/\t \subset \nu_c(L\g^*,\t)=\{c\}\times L\g^*/\t$ provides the desired splitting at the point $c \in \t$. 

When $\exp(a)$ is not central then $\ker(\d_a^*)=\t$ is too small. We claim that for all $c \in \Gamma$, $\{a\}\times \ker(\d_c^*)$ is a complement to the tangent space to the $\Omega G$ orbit through $a$. Since the codimension of the $\Omega G$ orbit through $a$ is $\dim(\g)$, it suffices to check that $\ker(\d_c^*) \cap \d_a(\Omega \g)=0$. Thus we check for solutions to the equation
\begin{equation} 
\label{e:transversezns}
\d_a\zeta=\xi, \qquad \zeta(0)=0, \quad \xi(s)=\Ad_{\exp(-sc)}(\xi_0),\quad \xi_0 \in \g_{\exp(c)}=\g.
\end{equation}
Applying $\Ad_{\exp(sc)}$ to both sides and setting $\tilde{\zeta}(s)=\Ad_{\exp(sc)}(\zeta)$, $\tilde{a}=\Ad_{\exp(sc)}(a)-c=a-c\ne 0$ yields
\[ \d_{\tilde{a}}\tilde{\zeta}=\xi_0 \quad \Rightarrow \quad \d_{\tilde{a}}(\tilde{\zeta}-\ad_{\tilde{a}}^{-1}\xi_0)=0. \]
Since $\tilde{a} \notin \Gamma$, \eqref{e:kerda} shows $\tilde{\zeta}-\ad_{\tilde{a}}^{-1}\xi_0 \in \t$, and in particular is constant. Since $\tilde{\zeta}(0)=0$, we obtain $\tilde{\zeta}=0$. Thus the only solution to \eqref{e:transversezns} is $\zeta=0$, verifying the claim.

Using the claim, we obtain the desired splitting of $\nu(L\g^*,\t)$ by using the subspace $\{c\}\times \ker(\d_c^*)/\t$ at lattice points $c \in \Gamma \subset \t$, and then simply interpolating between these splittings for $a \notin \Gamma$ (recall that convex combinations of splittings are again splittings). Interpolation can be done $N(T)\ltimes \Pi$-equivariantly by choosing a bump function $\chi \in C^\infty_c(\t)$ such that (i) $\supp(\chi)\cap \Gamma=\{0\}$, and (ii) the collection of translates of $\chi$ under elements of $\Gamma$ is a partition of unity.
\end{example}

\begin{definition}[\cite{LMSspinor}]
\label{d:globtrans}
Let $(\M,\omega_\M,\mu_\M)$ be a proper Hamiltonian $LG$-space, and let $R\subset L\g^*$ be as constructed in Proposition \ref{p:trans}. As $\mu_\M$ is $LG$-equivariant, the inverse image
\[ X=\mu_\M^{-1}(R) \]
is a finite dimensional $N(T)\ltimes \Pi$-invariant submanifold (with boundary) of $\M$ that we will refer to as a \emph{global transversal} of $\M$. The interior of $X$ (resp. $R$) is denoted $X^\circ$ (resp. $R^\circ$). The moment map $\mu_\M$ restricts to a map $\mu_X \colon X \rightarrow R$.
\end{definition}
\begin{remark}
\label{r:virtualfund}
In case the moment map $\mu_\M$ happens to be transverse to $\t \subset L\g^*$, one can work with the simpler choice $\mu_\M^{-1}(\t)$ instead of $\mu_\M^{-1}(R)$. We will in any case shortly introduce an operator on $X^\circ=\mu_\M^{-1}(R^\circ)$ that plays the role of a K-homological `virtual fundamental class' for $\mu_\M^{-1}(\t)$.
\end{remark}
We introduce some further notation and properties related to $X$:
\medskip

\noindent \emph{Riemannian metrics}. Under the map $\hol_\M$, the quotient $X/\Pi$ identifies with a compact $N(T)$-invariant submanifold with boundary of the compact Riemannian manifold $M=\M/\Omega G$, with interior $X^\circ/\Pi$ an open subset of $M$. Let $g_{X/\Pi}$ be the restriction of the $G$-invariant metric on $M$ to $X/\Pi$. Choose an $N(T)$-invariant boundary defining function $r$ on $X/\Pi$. Modifying the metric on $M$ if necessary, we may assume $g_{X/\Pi}$ takes a product form $\d r^2+g_{\partial X/\Pi}$ in a collar neighborhood of the boundary. Its pullback to $X$ is an $N(T)\ltimes \Pi$-invariant Riemannian metric $g_X$. We also introduce a b-metric (in the sense of Melrose \cite{melrose1993atiyah}) on $X/\Pi$, equal to $g_{X/\Pi}$ outside the collar, and given by $\d r^2/r^2+g_{\partial X/\Pi}$ on the collar. The pullback of the b-metric to $X$ will be denoted $g$. The restriction of $g$ to the interior $X^\circ$ is an ordinary Riemannian metric, which is complete with a cylindrical end. We identify $TX^\circ\simeq T^*X^\circ$ using $g$.

\medskip
\noindent \emph{Proper maps}. Let $(\mu,\nu_0)$ be the two components of the composition $\mu_X \colon X \rightarrow R \simeq \t \times \ol{B}_\epsilon(\t^\perp)$. Then $\mu$ is proper as the composition of the proper maps $\mu_X=\mu_\M|_X$ and $R \rightarrow \t$. Moreover $\mu|_{X^\circ}$ is the moment map for the $T$-action on $X^\circ$ (recalling that we identify $\t \simeq \t^*$ using the inner product `$\cdot$'). Similarly the restricted map $(\mu,\nu_0) \colon X^\circ \rightarrow R^\circ\simeq \t \times B_\epsilon(\t^\perp)$ is proper. Composing $\nu_0$ with an $N(T)$-equivariant diffeomorphism $B_\epsilon(\t^\perp)\simeq \t^\perp$, we obtain a proper $N(T)\ltimes \Pi$-equivariant map
\begin{equation} 
\label{e:muperp}
(\mu,\nu)\colon X^\circ \rightarrow \t \times \t^\perp.
\end{equation}
It is convenient to choose the diffeomorphism $B_\epsilon(\t^\perp)\simeq \t^\perp$ such that
\begin{equation}
\label{e:muperpdecay}
(1+|\nu(x)|^2)^{-1}|\d\nu(x)| \rightarrow 0 \quad \text{ as } \quad x \rightarrow \partial X,
\end{equation} 
and we assume this below (see for example \cite[Section 4.7]{LSQuantLG} or \cite[Section 3.3]{LSWittenDef} for further discussion).

\medskip
\noindent \emph{Spin-c structure}. Let $\scr{S}$ be a level $\ell>0$ spinor bundle for $\hol_\M^*TM$ (Definition \ref{d:spinormodule}). As $X^\circ/\Pi$ identifies with an open subset of $M$, the pullback of $\scr{S}$ to $X^\circ$ is a spinor module for the tangent bundle $(TX^\circ,g_X|_{X^\circ})$ that we still denote $\scr{S}$. Recall that we also defined a complete Riemannian metric $g$ on $X^\circ$. Given a spinor bundle $\c_1 \colon \bC l(TY,g_1)\rightarrow \End(\scr{S})$ on a Riemannian manifold $(Y,g_1)$, and given a new Riemannian metric $g_2$ on $Y$, one obtains a new Clifford action $\c_2=\c_1\circ (g_1^{-1}g_2)^{1/2}\colon \bC l(TY,g_2)\rightarrow \End(\scr{S})$. Thus $\scr{S}$ becomes becomes a spinor module for $\bC l(TX^\circ,g)$. The pullback of the anticanonical line bundle $\scr{L}$, $\nabla^{\scr{L}}$, and $\varpi$ to $X$ will be denoted with the same symbols. Define $\phi \colon X \rightarrow \t$ to be $\phi_\M|_X$ composed with projection to $\t$. By \eqref{e:levelellgauge}, $\phi$ is $N(T)\ltimes \Pi$-equivariant for the level $\ell$ action of $\Pi$ on $\t$:
\begin{equation}
\label{e:Piequivphi}
(\eta^*\phi)(x)=\phi(\eta \cdot x)=\phi(x)+\ell \eta
\end{equation} 
for all $\eta \in \Pi$. The assumption $\ell>0$ implies that $\phi \colon X \rightarrow \t$ is a proper map.

\subsection{Dirac operators on the transversal}\label{s:DiracTrans}
Let $X$, $\scr{S}$ be as in the previous subsection. Choose an $N(T)\ltimes \wh{\Pi}^{(\ell)}$-invariant Clifford connection $\nabla$ on $\scr{S}$, and let $\Dirac_{\scr{S}}$ be the corresponding Dirac operator, defined as the composition
\[ C^\infty(X^\circ,\scr{S})\xrightarrow{\nabla}C^\infty(X^\circ,T^*X^\circ\otimes \scr{S})\xrightarrow{\c}C^\infty(X^\circ,\scr{S}).\]
Completeness of the Riemannian metric $g$ on $X^\circ$ guarantees that the operator $\Dirac_{\scr{S}}$ with initial domain $C^\infty_c(X^\circ,\scr{S})$ is essentially self-adjoint in $L^2(X^\circ,\scr{S})$. Following Atiyah \cite{AtiyahKHomology} and Kasparov \cite{KasparovNovikov} (see also \cite{HigsonRoe}), $\Dirac_{\scr{S}}$ may be thought of as representing the fundamental class for $X^\circ$ in the analytic approach to K-homology.

As mentioned briefly in Remark \ref{r:virtualfund}, we will in fact work with an operator that may be thought of as representing the fundamental class of $\mu_\M^{-1}(\t)=\nu^{-1}(0)\subset X^\circ$. Since the subset $\nu^{-1}(0)$ is often singular, we implement this by twisting the Dirac operator $\Dirac_{\scr{S}}$ on $X^\circ$ by the pullback under $\nu \colon X^\circ \rightarrow \t^\perp$ of a Bott-Thom class $\scr{B} \in K_T(\t^\perp)$ in K-theory.

Let $(\scr{S}_{\t^\perp},c_{\t^\perp})$ denote the $T$-equivariant $\bZ_2$-graded spin representation for $\Cl(\t^\perp)$:
\begin{equation} 
\label{e:Stperp}
\scr{S}_{\t^\perp}=\wedge \n^- \otimes \bC_{\rho},\qquad \c_{\t^\perp}(\xi)=\sqrt{2}(\epsilon(\xi^{1,0})-\iota(\xi^{0,1}))\otimes 1 
\end{equation}
where $\n^- \subset \g_\bC$ is the direct sum of the negative root spaces, $\rho$ is the half sum of the positive roots, and we endow $\t^\perp$ with the complex structure $(\t^\perp)^{1,0}=\n^-$. The corresponding Bott-Thom element for $\t^\perp$ is the element $\scr{B}\in \sK^0_T(\t^\perp)$ represented by the morphism of vector bundles $\t^\perp\times \scr{S}_{\t^\perp}^+\rightarrow \t^\perp\times \scr{S}_{\t^\perp}^-$ given by $\c_{\t^\perp}(\xi)$ over the point $\xi \in \t^\perp$. This element is Weyl antisymmetric: under the transformation of $\sK^0_T(\t^\perp)$ induced by $w \in W$, $\scr{B}$ is multiplied by $(-1)^{l(w)}$ (see for example \cite[Proposition 4.8]{LSQuantLG}, except that in the latter we preferred not to include the $\bC_{\rho}$ shift of \eqref{e:Stperp}).
\begin{definition}
\label{d:virtualfund}
Let $S=\scr{S}_{\t^\perp}\wh{\boxtimes} \scr{S}$. The operator $\Dirac_{\scr{S}}$ trivially extends to act on sections of $S$. Define a new Dirac-type operator $\Dirac$ acting on sections of $S$ by
\[ \Dirac=\i \c_{\t^\perp}(\nu)\hotimes 1+\Dirac_{\scr{S}}.\]
In terms of a local orthonormal frame $e_1,...,e_n$ of $TX$ near $x \in X$, the operator $\Dirac$ is
\[ \Dirac(\tau\hotimes s)(x)=\i\c_{\t^\perp}(\nu(x))\tau(x)\hotimes s(x)+(-1)^{\deg(\tau)}\sum_{j=1}^n \partial_{e_j}\tau(x) \hotimes \c(e_j)s(x)+\tau(x) \hotimes \c(e_j)\nabla_{e_j}s(x).\] 
\end{definition}
The term involving $\c_{\t^\perp}(\nu)$ in Definition \ref{d:virtualfund} plays the role of a potential in the $\t^\perp$-directions, with $-\c_{\t^\perp}(\nu(x))^2=|\nu(x)|^2\rightarrow \infty$ as $x\in X^\circ$ approaches $\partial X$. 

\subsection{The Spin$_c$ Kirwan vector field}\label{s:Kirwanvf}
Throughout this section $\scr{S}$ denotes a level $\ell>0$ spinor bundle on $\M$, and $\phi \colon X \rightarrow \t$ denotes the (projection to $\t$) of the Spin$_c$ moment map for some choice of connection, as in Definition \ref{d:spincmoment}. In this section we prove some special properties of the Kirwan vector field associated to $\phi$. (Analogous properties hold for the Kirwan vector field associated to $\mu$.)
\begin{definition}
The \emph{Spin}$_c$ \emph{Kirwan vector field} is the $N(T)$-invariant vector field $\kappa$ on $X$ given at the point $x \in X$ by
\[ \kappa(x)=\phi(x)_{X}(x),\]
where $\phi(x)_X$ denotes the vector field on $X$ generated by $\phi(x)\in \t$. We also define
\[ \bar{\phi}=(1+\phi^2)^{-1/2}\phi, \qquad \bar{\kappa}(x)=\bar{\phi}(x)_X(x).\]
Clearly $\bar{\phi}$ is a bounded map. Since the actions of $T,\Pi$ on $X$ commute and $X/\Pi$ is compact, $\bar{\kappa}$ is a bounded vector field. 
\end{definition}
Let
\[ \Z=\kappa^{-1}(0_{X})=\bigcup_{\beta \in \B} \Z_\beta, \qquad \Z_\beta=X^\beta \cap \phi^{-1}(\beta) \]
where $\B$ is the infinite, discrete, $W$-invariant subset of $\beta \in \t$ such that $X^\beta \cap \phi^{-1}(\beta) \ne \emptyset$. By properness of $\phi$, each $\Z_\beta$ is compact. Let 
\[ Z=\Z\cap \nu^{-1}(0), \qquad Z_\beta=\Z_\beta \cap \nu^{-1}(0) \]
denote the intersection of $\Z$, $\Z_\beta$ with $\nu^{-1}(0)$. (Note that some of the $Z_\beta$ could be empty.) Although $\kappa$ is \emph{not} $\Pi$-invariant, the subset $\Z$ has, nevertheless, the following property: 
\begin{proposition}
\label{p:weakperiodicity}
There is a minimal \emph{finite} subset $\B_\ast\subset \B$ such that for each $\beta \in \B$ there is a $\beta_\ast \in \B_\ast$ and $\eta \in \Pi$ such that $\Z_\beta=\eta \cdot \Z_{\beta_\ast}$.
\end{proposition}
Before giving the proof, we introduce the notation 
\begin{equation}
\label{e:stablist}
\{\t_i\subset \t\}_{i \in \I},
\end{equation} 
for the list of rational subspaces that arise as infinitesimal stabilizers of subsets of $X$. Since the actions of $T$, $\Pi$ on $X$ commute and since $X/\Pi$ is compact, the set $\I$ is \emph{finite}.
\begin{proof}
For $i \in \I$, let $\B_i \subset \B$ be the subset of $\beta$ such that $X^\beta=X^{\t_i}$. Thus $\cup \B_i=\B$, and as $\beta$ ranges over $\B_i$, $X^\beta=X^{\t_i}$ does not vary. Therefore it suffices to show that the image $\exp(\B_i)$ of each $\B_i$ under the quotient map $\exp \colon \t \rightarrow \t/\Pi=T$ is finite.

For an affine subspace $\Delta \subset \t$, let $\Delta_0$ be the unique subspace parallel to $\Delta$. Note that since the inner product $\cdot$ on $\t$ is integral, i.e. $\Pi \cdot \Pi \subset \bZ$, a subspace $\Delta_0$ is rational if and only if its orthogonal complement $\Delta_0^\perp$ is.

The map $\phi$ satisfies $\d \pair{\phi}{\xi}=-\iota(\xi_X)\varpi$, and it follows from this equation that the image under $\phi$ of each connected component of the fixed-point set $X^{\t_i}$ is contained in an affine subspace $\Delta$ with $\Delta_0^\perp=\t_i$.  Let $\S_i$ denote the collection of affine subspaces $\Delta$ (each a translate of $\t_i^\perp$) that arise in this way from a connected component of $X^{\t_i}$. Note that $\Pi$ acts naturally on $\S_i$, and since $X/\Pi$ is compact, the set of cosets of the $\Pi$ action is finite; let $s_i$ be its cardinality. 

If $\beta \in \B_i$ then there exists a $\Delta \in \S_i$ such that $\beta=\pr_\Delta(0)$, the orthogonal projection of $0$ onto $\Delta$. Now suppose $\Delta \in \S_i$ and $\Delta^\prime=\Delta+\eta$ for some $\eta \in \Pi$.  Then $\Delta_0^\perp=(\Delta'_0)^\perp=\t_i$ and
\[ \pr_{\Delta^\prime}(0)=\pr_\Delta(0)+\pr_{\t_i}(\eta),\]
where $\pr_{\t_i}$ denotes orthogonal projection to $\t_i$.  Thus when $\Delta$ is translated by some $\eta \in \Pi$, the projection $\pr_\Delta(0)$ changes by an element of $\pr_{\t_i}(\Pi)$. By integrality of the inner product, $\pr_{\t_i}(\Pi)$ is a lattice in $\t_i$.  On the other hand, since $\t_i$ is integral, $\t_i \cap \Pi$ is a full-rank lattice in $\t_i$, hence has some finite index $n_i$ in $\pr_{\t_i}(\Pi)$. It follows that $\# \exp(\B_i)\le s_i\cdot n_i$. (Thus in fact we have the bound $\# \B_\ast \le \sum_{i \in \I} s_i\cdot n_i$.)
\end{proof}

For each $\beta_* \in \B_*$, let $\U_{\beta_*}$ be an open $T$-invariant neighborhood of $\Z_{\beta_*}$ in the manifold with boundary $X$. Using translations by elements of $\Pi$, we obtain open neighborhoods $\U_\beta$ of $\Z_\beta$ for each $\beta \in \B$. Making the $\U_{\beta_*}$ smaller if necessary, we can ensure that closure of the images $\phi(\U_\beta)$ are pair-wise disjoint. Let
\[ \U=\bigcup_{\beta \in \B} \U_\beta.\]
Let $\d(\cdot,\cdot)$ be the topological metric on $X$ determined by the Riemannian metric $g_X$. Let $\d_\t(\cdot,\cdot)$ denote the metric induced by the norm on $\t$.
\begin{lemma}
\label{l:distanceU}
There is an $\epsilon>0$ such that for all $i \in \I$ (see \eqref{e:stablist}),
\[ \d(x,X^{\t_i})<\epsilon, \quad \d_\t(\phi(x),\t_i)<\epsilon \quad \Rightarrow \quad x \in \U.\]
\end{lemma}
\begin{proof}
First note that there is a $\epsilon'>0$ such that for all $\beta \in \B$, 
\begin{equation} 
\label{e:epsilonprime}
\d(x,X^\beta)<\epsilon', \quad \d_\t(\phi(x),\beta)<\epsilon' \quad \Rightarrow \quad x \in \U.
\end{equation}
Indeed it is clear that we can find an $\epsilon'>0$ that works for $\beta$ in the finite set $\B_*$, and it then works for all $\beta \in \B$ by $\Pi$-invariance.

Suppose the statement is false. Then there exists an $i \in \I$ and a sequence $\{x_n\}$ in $\ol{X} \backslash \U$ such that $\d(x_n,X^\h), \d_\t(\phi(x_n),\h) \rightarrow 0$, where $\h=\t_i$. Passing to a subsequence, we can assume $x_n\Pi \rightarrow x\Pi$ in the compact space $X/\Pi$. Necessarily $x\Pi \in (X/\Pi)^\h=X^\h/\Pi$. Let $y_n \in X$ be such that $y_n\Pi=x\Pi$ and $\d(x_n,y_n)$ is minimal. Then $\d(x_n,y_n)\rightarrow 0$ and $y_n \in X^\h$. Note also that by $\Pi$-equivariance, $\phi$ is Lipschitz for some constant $C\ge 1$. Therefore
\begin{equation} 
\label{e:tdistconv}
\d_\t(\phi(y_n),\h)\le \d_\t(\phi(y_n),\phi(x_n))+\d_\t(\phi(x_n),\h) \le C\cdot \d(y_n,x_n)+\d_\t(\phi(x_n),\h)\rightarrow 0. 
\end{equation}

Let $X^\h_j$, $j \in \J$ be the connected components of $X^\h$. There is a finite subset $\J_* \subset \J$ such that each $X^\h_j$ is of the form $\eta \cdot X^\h_{j_*}$ for some $j_*\in \J_*$ and $\eta \in \Pi$. The moment map sends each $X^\h_j$ into a closed subset of an affine hyperplane parallel to the orthogonal complement $\h^\perp$. By $\Pi$-periodicity and because $\h$ is rational, there is a constant $c>0$ such that for all $j \in \J$, either $\phi(X_j^\h)\cap \h \ne \emptyset$, or else $\d_\t(\phi(X^\h_j),\h)>c$. In particular \eqref{e:tdistconv} implies that $\{y_n\}$ is eventually contained in the union of those components $X_j^\h$ such that $\phi(X_j^\h) \cap \h\ne \emptyset$. Then by \eqref{e:tdistconv}, there exists $n_0$ and $\beta \in \phi(X^\h)\cap \h\subset \B\cap \h$ such that 
\[ \d_\t(\phi(y_{n_0}),\beta)<\frac{\epsilon'}{2},\qquad \d(x_{n_0},y_{n_0})<\frac{\epsilon'}{2C}. \]
These inequalities, together with $y_{n_0} \in X^\h\subset X^\beta$, imply $\d(x_{n_0},X^\beta)<\epsilon'/2C<\epsilon'$, $\d_\t(\phi(x_{n_0}),\beta)<\epsilon'$, and hence \eqref{e:epsilonprime} yields $x_{n_0} \in \U$, a contradiction.
\end{proof}
\begin{proposition}
\label{p:Kirwanpos}
There is a constant $c>0$ such that $|\kappa(x)|\ge c$ for all $x \in X\backslash \U$.
\end{proposition}
\begin{proof}
Suppose the statement is false. Then there is a sequence $\{x_n\}$ in $X\backslash \U$ such that $|\kappa(x_n)|\rightarrow 0$. Passing to a subsequence, we can assume that $x_n\Pi \rightarrow x\Pi$ in the compact space $X/\Pi$. Let $\h \subset \t$ be maximal such that $x\Pi \in (X/\Pi)^\h=X^\h/\Pi$ (thus $\h=\t_i$ for some $i \in \I$, see \eqref{e:stablist}). Let $\phi=\phi_\h+\phi_{\h^\perp}$ be the decomposition of $\phi$ into its $\h$, $\h^\perp$ components. Since $\d(x_n,X^\h)\rightarrow 0$, to obtain a contradiction it suffices, by Lemma \ref{l:distanceU}, to show that $\d_\t(\phi(x_n),\h)=|\phi_{\h^\perp}(x_n)|\rightarrow 0$. 

Let $\varrho \colon X\times \t \rightarrow TX$ be the infinitesimal action. We will use the same symbol for the infinitesimal action on $X/\Pi$. By the slice theorem for actions of compact Lie groups, there is a compact neighborhood $V$ of $x\Pi$ in $X/\Pi$, a constant $\delta>0$ and a smooth subbundle $A$ such that
\[ T(X/\Pi)|_V=\varrho(V\times \h^\perp)\oplus A, \qquad \varrho(V\times \h)\subset A,\]
and
\begin{equation} 
\label{e:bdedbelowonhperp}
|\varrho_{x'\Pi}(\xi)|>\delta|\xi|, \qquad x'\Pi \in V,\quad \xi \in \h^\perp.
\end{equation}
Let $\theta \in (0,\pi/2]$ be the minimal angle (with respect to the Riemannian metric $g_{X/\Pi}$ on $X/\Pi$) between $\varrho(V\times \h^\perp)$, $A$ over the compact set $V$.

The Kirwan vector field 
\[ \kappa(x_n)=\varrho_{x_n}(\phi(x_n))=\varrho_{x_n}(\phi_\h(x_n))+\varrho_{x_n}(\phi_{\h^\perp}(x_n)).\]
For $n$ sufficiently large, $x_n\Pi \in V$, and then by elementary geometry
\[ |\kappa(x_n)|\ge |\varrho_{x_n}(\phi_{\h^\perp}(x_n))|\sin(\theta). \]
Since $|\kappa(x_n)|\rightarrow 0$, we get $|\varrho_{x_n}(\phi_{\h^\perp}(x_n))|\rightarrow 0$. By \eqref{e:bdedbelowonhperp}, $|\phi_{\h^\perp}(x_n)|\rightarrow 0$, as desired.
\end{proof}

\section{Admissible K-theory classes and the analytic index}\label{s:quasi}
In this section we explain the admissibility condition on K-theory classes alluded to in the introduction, construct the index map, and prove the non-abelian localization formula. Throughout this section $\M$ is a proper Hamiltonian $LG$-space, $X\subset \M$ is a global transversal with spinor bundle $\scr{S}$ at level $\ell>0$.

\subsection{K-theory}
Let $H$ be a compact Lie group and let $Y$ be an $H$-space. For us the $H$-equivariant K-theory of $Y$, denoted $K_H(Y)$, refers to the non-compactly supported $H$-equivariant $0$-th topological K-theory group of $Y$. A standard description of $K_H(Y)$ is as the abelian group of homotopy classes of continuous maps from $Y$ to the space of Fredholm operators on $L^2(H)\otimes \ell^2(\bZ)$ equipped with the norm topology, cf. \cite[Section 5]{segal1970fredholm}. 

It will be convenient to represent K-theory classes using pairs $(\E,Q)$, where $\E=\E^+\oplus \E^-$ is an $H$-equivariant $\bZ_2$-graded Hilbert bundle over $Y$ (with structure group the unitary group carrying the norm topology) and $Q$ is a family of possibly unbounded odd self-adjoint operators on the fibres of $\E$, such that $F=Q(1+Q^2)^{-1/2}$ is a continuous section of $\scr{B}(\E)$ (bounded operators on the fibres of $\E$) and $1-F^2=(1+Q^2)^{-1}$ is a continuous section of $\scr{K}(\E)$ (compact operators on the fibres of $\E$). Given such a pair $(\E,Q)$, the corresponding continuous map from $Y$ to the space of Fredholm operators is obtained by first trivializing $\E$ (after stabilization if necessary) and then taking the component $F^+$ of $F$ that maps $\E^+$ to $\E^-$. See \cite{HigsonPrimer, KasparovNovikov, FHTII} for further context on representing K-theory classes with data of this sort.

\subsection{Admissible cycles}\label{s:qpcycles}
\begin{definition}
Let $\E$ be a $T$-equivariant Hilbert bundle over $X$. We say that $\E$ is $T$-\emph{finite} if $\E$ can be realized as a subbundle of a trivial $T$-equivariant Hilbert bundle $X \times \E_0$ such that $\E_0$ has finitely many $T$-isotypical components.
\end{definition}
\begin{remark}
$\E_0$ has finitely may $T$-isotypical components if and only if the map $T\rightarrow \scr{U}(\E_0)$ is continuous for the norm topology. Similarly $\E$ is $T$-finite if and only if $T \rightarrow \Aut(\E)$ is continuous for the norm topology.
\end{remark}

We will specialize to $T\times \Pi$-equivariant $T$-finite smooth Hilbert bundles $\E \rightarrow X$. Since $T \times \Pi$ acts properly on $X$, it is always possible to construct smooth $T\times \Pi$-invariant Hermitian connections on $\E$. Two such connections differ by a smooth $T\times \Pi$-invariant section of $\mf{u}(\E)\subset \scr{B}(\E)$, the bundle of bounded skew-adjoint endomorphisms of $\E$.
\begin{definition}
\label{d:mmH}
Let $(\E,\nabla^\E)$ be a $T\times \Pi$-equivariant $T$-finite smooth Hilbert bundle with $T\times \Pi$-invariant Hermitian connection. Define the \emph{moment map} $\phi_\E$ of the pair $(\E,\nabla^\E)$ (cf. \cite{BerlineGetzlerVergne} in the finite dimensional case) by
\begin{equation}
\label{e:mmH}
2\pi\i\pair{\phi_\E}{\xi}=\L^\E_\xi - \nabla^\E_{\xi_X}, \qquad \xi \in \t.
\end{equation}
Note that $\pair{\phi_\E}{\xi}\in C^\infty(X,\scr{B}(\E))$ is bounded because of the $T$-finiteness condition, hence $\phi_\E \in \t^*\otimes C^\infty(X,\scr{B}(\E))$ and is $T\times \Pi$-invariant. 
\end{definition}

\begin{definition}
\label{d:admissible}
Let $\E \rightarrow X$ be a $T$-equivariant $\bZ_2$-graded smooth Hilbert bundle with $T$-invariant Hermitian connection $\nabla^\E$, and let $Q=\{Q_x\}_{x \in X}$ be a $T$-equivariant family of odd unbounded self-adjoint operators on the fibres $\{\E_x\}_{x \in X}$ of $\E$. The triple $(\E,\nabla^\E,Q)$ will be called an \emph{admissible cycle} if
\begin{enumerate}
\item For each $x \in X$, $(1+Q_x^2)^{-1}$ is a compact operator.
\item $(\E,\nabla^\E)$ is $T \times \Pi$-equivariant and $\E$ is $T$-finite.
\item The family $Q$ is smooth in the sense that for any smooth compactly supported section $s$ of $\E$ such that $s_x \in \dom(Q_x)$ for all $x \in X$, the compactly supported section $Qs$ of $\E$ is again smooth.
\item $\nabla^\E$-parallel transport maps the subset $\dom(Q_\bullet)=\bigsqcup_{x \in X}\dom(Q_x)\subset \E$ to itself, and the total $\nabla^\E$-covariant derivative of $Q$ is a smooth bounded section of $T^*X\otimes \scr{B}(\E)$.
\end{enumerate}
\end{definition}
\ignore{
\begin{remark}
\texttt{To think about further}: It seems to me likely that one could reduce to this scenario from a simpler and more general starting point. A few preliminary thoughts. Start with a smooth self-adjoint bounded Kasparov representative $F$ acting on a smooth trivial Hilbert bundle $\H$ with fiber $H$. Choose an unbounded self-adjoint linear map $r \colon H'=\dom(r) \rightarrow H$ such that $r^{-1}$ exists (so $r$ is bounded below by a strictly positive constant) and is compact. Try replacing $F$ with $Q=Fr$ so $(1+Q^*Q)^{-1}=(1+rF^2r)^{-1}=r^{-1}(r^2+F^2)^{-1}r^{-1}$ is certainly. This $Q$ is not self-adjoint, but I believe using the usual trick one can make it so. Its domain bundle is preserved by the trivial connection. So it looks like roughly we want the derivative of $F$ to be compact in a uniform manner, so that, working backwards, $r$ can be chosen appropriately: we need $(\nabla F)r$ to be a bounded section of $\scr{B}(\E)$, basically. If gradient of $F$ is unbounded, we're stuck??

Another comment is that it can be difficult to control the gradient of $F$. For example if $F$ acts on a finite dimensional vector bundle and is $f\Id$ for some function $f$, then we cannot pick a connection to arrange that the gradient is small.
\end{remark}
}
\begin{remark}
\label{r:smoothness}
Condition (d) requires further explanation. The condition that $\nabla^\E$-parallel transport maps $\dom(Q_\bullet)$ into itself means that we may simultaneously trivialize $\E$, $\dom(Q_\bullet)$ on geodesic balls $B_{x_0}$ in $X$ by parallel translation along radial geodesics (in particular observe that $\dom(Q_\bullet)$ is a subbundle of $\E$). Thus $Q|_{B_{x_0}}$ can be viewed as a family of unbounded self-adjoint operators on a fixed Hilbert space $\E_{x_0}$ having the same domain $\dom(Q_{x_0})$. A functional analytic argument \cite[p.549]{kriegl1997convenient} shows that the seemingly weak smoothness condition (c) already implies that the corresponding map 
\begin{equation} 
\label{e:Qtrivialized}
Q|_{B_{x_0}}\colon B_{x_0} \rightarrow \scr{B}(\dom(Q_{x_0}),\E_{x_0}) 
\end{equation}
is smooth, where $\dom(Q_{x_0})$ is equipped with the graph norm. Consequences of this include that the resolvents $(Q\pm \i)^{-1}$ vary smoothly, and also that a section $s$ of $\dom(Q_\bullet)|_{B_{x_0}}$ (topologized with the graph norm of $Q_{x_0}$) is smooth if and only if $s$ is smooth as a section of $\E|_{B_{x_0}}$ (because $s=(Q+\i)^{-1}(Q+\i)s$). The total $\nabla^\E$-covariant derivative of $Q$ at the point $x_0$ can be defined to be the derivative of the map \eqref{e:Qtrivialized} at the point $x_0$. The result is a priori an element of $T^*_{x_0}X\otimes \scr{B}(\dom(Q_{x_0}),\E_{x_0})$, and part of condition (d) is that we require it to extend continuously to an element of $T^*_{x_0}X\otimes \scr{B}(\E_{x_0})$. The second part of condition (d) is that the section of $T^*X\otimes \scr{B}(\E)$ obtained by assembling the total derivative at all points $x_0 \in X$ should be smooth and bounded.
\end{remark}
\begin{remark}
Notice that we \emph{do not} assume $Q$ is $\Pi$-equivariant. Example \ref{ex:dbarondisk} below shows that it is even possible that $\dom(Q_\bullet)$ fails to be $\Pi$-invariant.
\end{remark}

The smoothness assumption on $Q$ implies (a fortiori) that $F=Q(1+Q^2)^{-1/2}$ varies continuously in the norm topology, and hence an admissible cycle represents a K-theory class. 
\begin{definition}
A class $\sf{E}\in K_T(X)$ is \emph{admissible} if there is an admissible cycle $(\E,\nabla^\E,Q)$ such that the pair $(\E,Q)$ represents $\sf{E}$. A class $\sf{E} \in K_T(\M)$ is \emph{admissible} if its pullback to $X$ is admissible. We use the notation
\[ K_T^\ad(X)\subset K_T(X), \qquad K_T^\ad(\M)\subset K_T(\M) \]
for the subsets of admissible classes. It is immediate from the definition that these are $R(T)$-subalgebras. An element of $K_G(\M)$ is \emph{admissible} if its image in $K_T(\M)$ is admissible, and such classes likewise form an $R(G)$-subalgebra of $K_G(\M)$.
\end{definition}
\begin{example}
As already mentioned in the introduction, any K-theory class represented by a finite rank $LG$-equivariant vector bundle on $\M$ is admissible.
\end{example}
\begin{example}
\label{ex:dbarondisk}
Let $\Sigma=\bD$ be the disk with its standard metric an complex structure, and let $\M_\bD\simeq LG/G$ the corresponding Hamiltonian $LG$-space. Let $V$ be a finite-dimensional representation of $G$. For $A \in \A_{fl,\bD}$ let 
\[ \bar{\partial}_A=\bar{\partial}+A^{0,1} \colon L^{2,1}_{<0}(\bD,V)\rightarrow L^2(\bD,V\otimes \wedge^1 T^{0,1}\bD) \]
be the corresponding $\bar{\partial}$-operator, with domain $L^{2,1}_{<0}(\bD,V)$ consisting of $V$-valued $L^{2,1}$ functions $f$ on $\bD$ such that $f|_{\partial \bD}$ has only negative Fourier modes. The family $\bar{\partial}_A+\bar{\partial}_A^*$, $A \in \A_{fl,\bD}$ descends to a family of Fredholm operators $Q$ over $\M_\bD=\A_{fl,\bD}/\G_{\bD,\partial \bD}$ acting on sections of a Hilbert bundle $\E$. In the sequel to this article we prove that the class in $K_G(\M_\bD)$ represented by $(\E,Q)$ is admissible, and that it is not induced by any $LG$-equivariant vector bundle on $\M_\bD$.
\end{example}
\begin{example}
Let $L\rightarrow \M$ be a level $(1,...,1)$ line bundle. Then the K-theory class represented by $L$ is \emph{not} admissible. To see this let $x \in X$ be fixed by some 1-parameter subgroup $T'\subset T$, and let $\lambda'$ be the weight for the action of $T'$ on $L|_x$. For each $\eta \in \Pi$ there is a restriction map $r_\eta\colon K_T(\M)\rightarrow R(T')$ given by pullback to the point $\eta \cdot x$. By \eqref{e:commutation}, $r_\eta(L)=\lambda'+\eta'$ where $\eta'=\eta|_{T'}$. The set $\{\lambda'+\eta'\mid \eta \in \Pi\}$ is an unbounded subset of the weight lattice of $T'$. This cannot occur for an admissible class by the $T$-finiteness condition.
\end{example}

\subsection{The analytic index of admissible cycles}\label{s:analyticindex}
Given a Hermitian connection $\nabla^\E$ on a Hilbert bundle $\E \rightarrow X$, let $\DiracE=1\hotimes_{\nabla^\E}\Dirac$ be the unbounded self-adjoint operator on $L^2(X^\circ,\E\hotimes S)$ obtained by coupling $\Dirac$ to $\E$ using $\nabla^\E$ (for self-adjointness of Dirac operators on complete Riemannian manifolds coupled to Hilbert bundles, cf. \cite[Proposition 1.16]{Ebert2016Index}).

\begin{lemma}
\label{l:TfiniteNorm}
Let $(\E,\nabla^\E)$ be a $T \times \Pi$-equivariant Hilbert bundle with Hermitian connection over $X$, where $\E$ is $T$-finite. Let $\chi \in C^\infty_c(X)^T$ and let $f \in C_0(\bR)$. Then $\|\chi f(\DiracE)_{[\lambda]}\|\rightarrow 0$ as $|\lambda|\rightarrow \infty$, $\lambda \in \Lambda$.
\end{lemma}
\begin{proof}
A function $f \in C_0(\bR)$ can be approximated uniformly in norm by functions of the form $g(s)(\i+s)^{-1}$, $g \in C_0(\bR)$. The properties of the potential term $c_{\t^\perp}(\nu)$ and in particular equation \eqref{e:muperpdecay} (see \cite[Section 4.7]{LSQuantLG} for detailed discussion) imply that the operator $\chi f(\DiracE)$ can be approximated in norm by operators of the form $\chi' f(\DiracE)$, where $\chi' \in C^\infty_c(X^\circ)^T$. Therefore we may reduce to the case where $f(s)=(\i+s)^{-1}$ and $\chi \in C^\infty_c(X^\circ)^T$. Then $f(\DiracE)$ is the resolvent of $\DiracE$, which has range contained in $L^{2,1}(X^\circ,\E\hotimes S)$ by ellipticity, and $\chi f(\DiracE)$ has range contained in $L^{2,1}(K,\E\hotimes S)$ where $K=\supp(\chi)\subset X^\circ$ is compact. Let
\[ \iota_\lambda^\E \colon L^{2,1}(K,\E\hotimes S)_{[\lambda]}\hookrightarrow L^2(K,\E\hotimes S)_{[\lambda]}, \qquad \iota_\lambda \colon L^{2,1}(K,S)_{[\lambda]}\hookrightarrow L^2(K,S)_{[\lambda]} \]
be the inclusions. Note that the norms 
\begin{equation}
\label{e:iotalambda}
\|\iota_\lambda\|\rightarrow 0 \quad  \text{as} \quad  |\lambda|\rightarrow \infty
\end{equation} 
by the Rellich lemma, because $\iota_\lambda$ are the isotypical components of the compact embedding $\iota \colon L^{2,1}(K,S)\hookrightarrow L^2(K,S)$. 

Likewise the $\iota^\E_\lambda$ are the isotypical components of an inclusion map $\iota^\E$, but unlike $\iota$, $\iota^\E$ need not be a compact embedding since $\E$ is allowed to have infinite rank. We claim that nevertheless the norm $\|\iota_\lambda^\E\| \rightarrow 0$ as $|\lambda|\rightarrow \infty$. Indeed we can argue locally, and we may use any connection, since the Sobolev space $L^{2,1}(K,\E\hotimes S)$ does not depend on the choice. Therefore assume $\E=X \times \E_0$ is the trivial $T$-equivariant Hilbert bundle with the trivial connection, and $\DiracE=\id_{\E_0}\hotimes \Dirac$. By $T$-finiteness, $\E_0$ has finitely many non-zero $T$-isotypical components $\E_{0,\lambda'}$, $\lambda' \in \Lambda'$. Then
\[ L^{2,1}(K,\E\hotimes S)_{[\lambda]}=\bigoplus_{\lambda'\in \Lambda'} \E_{0,\lambda'}\hotimes L^{2,1}(K,S)_{[\lambda-\lambda']} \]
and hence
\[ \|\iota_{\lambda}^\E\|=\sup_{\lambda' \in \Lambda'}\|\iota_{\lambda-\lambda'}\|.\]
The result follows from \eqref{e:iotalambda} and the finiteness of $\Lambda'$.
\end{proof}

A family of operators $Q$ as in Definition \ref{d:admissible} determines an odd essentially self-adjoint operator on $L^2(X^\circ,\E\hotimes S)$ with domain $C^\infty_c(X^\circ,\dom(Q_\bullet)\hotimes S)$. Its closure is a self-adjoint operator, also denoted $Q$ when there is no risk of confusion. 
\begin{proposition}
Let $(\E,\nabla^\E,Q)$ be an admissible cycle. Then the sum
\begin{equation}
\label{e:DiracQ}
\DiracQ=Q+\DiracE 
\end{equation}
is self-adjoint on $\dom(Q)\cap \dom(\DiracE)=\dom(Q)\cap L^{2,1}(X^\circ,\E\hotimes S)$.
\end{proposition}
\begin{proof}
We will deduce this as a very special case of the abstract functional analytic result \cite[Theorem 1.1]{lesch2019sums}. Let
\[ \mf{D}(Q,\DiracE)=\{v \in \dom(Q)\cap \dom(\DiracE)\mid Q v \in \dom(\DiracE), \DiracE v\in \dom(Q)\} \subset L^2(X^\circ,\E\hotimes S) \]
be the initial domain of the graded commutator $[Q,\DiracE]$. Then according to \emph{loc. cit.}, and since $C^\infty_c(X^\circ,\E\hotimes S)$ is a core for $\DiracE$, it suffices to verify that: $C^\infty_c(X^\circ,\dom(Q_\bullet)\hotimes S)\subset \mf{D}(Q,\DiracE)$, the graded commutator $[Q,\DiracE]$ is bounded, and $(Q\pm \i)^{-1}(C_c^\infty(X^\circ,\E\hotimes S)) \subset C^\infty_c(X^\circ,\dom(Q_\bullet)\hotimes S)$.

By condition (c) in the definition of admissible cycles $Q(C^\infty_c(X^\circ,\dom(Q_\bullet))\subset C^\infty_c(X^\circ,\E)$, while by condition (d), $\nabla^\E(C^\infty_c(X^\circ,\dom(Q_\bullet)))\subset C^\infty_c(X^\circ,T^*X\otimes \dom(Q_\bullet))$. It follows that $C^\infty_c(X^\circ,\dom(Q_\bullet))\subset \mf{D}(Q,\DiracE)$. Boundedness of $[Q,\DiracE]$ is a consequence of boundedness of $[\nabla^\E,Q]$, which in turn follows from condition (d) in the definition. The resolvent $(Q\pm \i)^{-1}\colon \E\rightarrow \dom(Q_\bullet)$ is smooth as we noted in Remark \ref{r:smoothness}, hence $(Q\pm \i)^{-1}(C_c^\infty(X^\circ,\E\hotimes S)) \subset C^\infty_c(X^\circ,\dom(Q_\bullet)\hotimes S)$.
\end{proof}

\begin{theorem}
\label{t:compactresolvent}
Let $(\E,\nabla^\E,Q)$ be an admissible cycle, and let $\DiracQ=\DiracE+Q$ be the corresponding unbounded self-adjoint operator on $L^2(X^\circ,\E\hotimes S)$. For each $\lambda \in \Lambda$ the operator $\DiracQ_{[\lambda]}$ has compact resolvent.
\end{theorem}
\begin{proof}
We must show that $(1+\DiracQs)^{-1}_{[\lambda]}$ is compact. Since
\[ (1+Q^2+\DiracEs)^{-1}-(1+\DiracQs)^{-1}=(1+\DiracQs)^{-1}[Q,\DiracE](1+Q^2+\DiracEs)^{-1} \]
and $[Q,\DiracE]$ is bounded, it suffices to show that $(1+Q^2+\DiracEs)^{-1}$ is compact. By positivity (see \cite[Proposition 1.4.5]{PedersenBook}; we learned of this trick from \cite{MohsenWitten}) there is a factorization
\[ (1+Q^2+\DiracEs)^{-1}=b(1+Q^2)^{-1/4}(1+\DiracEs)^{-1/4}b' \]
for some bounded operators $b,b'$. The operator $(1+Q^2)^{-1/4}$ is a bounded $T$-invariant section of the bundle of compact operators $\scr{K}(\E \hotimes S)$. The result therefore follows if we prove the following: for any continuous bounded $T$-invariant section $a$ of $\scr{K}(\E\hotimes S)$, the operator $a(1+\DiracEs)_{[\lambda]}^{-1/4}$ on $L^2(X^\circ,\E\hotimes S)_{[\lambda]}$ is compact.

The function $s \in \bR \mapsto (1+s^2)^{-1/4} \in (0,\infty)$ is a uniform limit of Schwartz functions having compactly supported Fourier transform. Since the compact operators form a closed ideal, it suffices to show that $ah(\DiracE)_{[\lambda]}$ is compact for any Schwartz function $h \in C^\infty(\bR)$ having compactly support Fourier transform. For such functions $h(\DiracE)$ is an operator of finite propagation (cf. \cite{HigsonRoe}): if $s \in C^\infty_c(X^\circ,\E\hotimes S)$ then $h(\DiracE)s$ has support contained in a neighborhood of $\supp(s)$ of size $r$, where $\supp(\hat{h})\subset [-r,r]$.

Let $\chi_\t \in C^\infty_c(\t)$ be a compactly supported bump function such that $\{\eta\cdot \chi_\t|\eta \in \Pi\}$ is a partition of unity. Let $\chi=\mu^*\chi_\t$ be the pullback. The set $\{\chi_\eta=\eta\cdot \chi|\eta \in \Pi\}$ is a partition of unity on $X$, and therefore
\begin{equation} 
\label{e:hDirac}
ah(\DiracE)=\sum_{\eta \in \Pi} \chi_\eta ah(\DiracE),
\end{equation}
the sum converging in the weak topology. Using the properness of $(\mu,\nu)\colon X^\circ \rightarrow \t \times \t^\perp$, the properties of the potential term $\c_{\t^\perp}(\nu)$ (in particular \eqref{e:muperpdecay}; see \cite[Section 4.7]{LSQuantLG} for details), and the compactness of $a$ on the fibres, the Rellich lemma implies that each of the summands in \eqref{e:hDirac} is compact. Thus the result follows if we show that \eqref{e:hDirac} converges in norm when restricted to the $\lambda$-isotypical subspace $L^2(X^\circ,\E\hotimes S)_{[\lambda]}$. Since a finite sum of compact operators is compact, it is enough to prove this for the sum over a finite index sublattice $\Pi'\subset \Pi$. Since $h(\DiracE)$ has finite propagation, we may choose $\Pi'$ such that the operators $\chi_\eta ah(\DiracE)$ have disjoint supports for $\eta \in \Pi'$. Then proving convergence of the sum in norm on $L^2(X^\circ,\E\hotimes S)_{[\lambda]}$ is the same as proving that $\chi_\eta\big(ah(\DiracE)\big)_{[\lambda]}\rightarrow 0$ in norm as $|\eta| \rightarrow \infty$. Choose lifts $\wh{\eta} \in \wh{\Pi}^{(\ell)}$ and let $a_\eta=\wh{\eta}^{-1} a \wh{\eta}$. Then using $\wh{\Pi}^{(\ell)}$-equivariance of $\DiracE$,
\[ \Big(\chi_\eta ah(\DiracE)\Big)_{[\lambda]}=\Big(\wh{\eta}\big(\chi a_\eta h(\DiracE)\big)\wh{\eta}^{-1}\Big)_{[\lambda]}=\wh{\eta}\Big(\chi a_\eta h(\DiracE)\Big)_{[\lambda+\ell \eta]}\wh{\eta}^{-1} \]
where in the second equality we used the commutation relation \eqref{e:commutation}. The norm of the operator in the last expression is at most
\[ \|a\|_\infty\cdot \|\chi h(\DiracE)_{[\lambda+\ell\eta]}\|. \]
Since $\ell>0$, $|\lambda+\ell \eta| \rightarrow \infty$ as $|\eta|\rightarrow \infty$. By Lemma \ref{l:TfiniteNorm}, $\|\chi h(\DiracE)_{[\lambda+\ell\eta]}\| \xrightarrow{|\eta|\rightarrow \infty}0$.
\end{proof}
In particular $\DiracQ_{[\lambda]}$ is Fredholm, allowing us to make the following definition.
\begin{definition}
The \emph{equivariant analytic index} of the admissible cycle $(\E,\nabla^\E,Q)$ is the element of $R^{-\infty}(T)$ given by
\[ \index_T(\DiracQ)=\sum_{\lambda \in \Lambda} \index(\DiracQ_{[\lambda]})\cdot e^\lambda.\]
\end{definition}

\subsection{Non-abelian localization}
In Section \ref{s:Kirwanvf} we introduced the Spin$_c$ Kirwan vector field $\kappa$ and its bounded version $\bkappa$ associated to $\phi$, $\bar{\phi}=\phi(1+\phi^2)^{-1/2}$ respectively. In Section \ref{s:transversal} we introduced a map $\nu \colon X^\circ \rightarrow \t^\perp$ with the property that $(\mu,\nu)\colon X^\circ \rightarrow \t \times \t^\perp$ is a proper map. Let $\DiracQ=\DiracE+Q$ be the operator on $L^2(X^\circ,\E\hotimes S)$ associated to an admissible cycle $(\E,\nabla^\E,Q)$ as in Section \ref{s:analyticindex}. We will study the deformation
\begin{equation} 
\label{e:DiracQs}
\DiracQ_s=\DiracQ+s\bar{\Phi}, \qquad \bar{\Phi}=\i(\c_{\t^\perp}(\nu)\hotimes 1-1\hotimes \c(\bkappa)),
\end{equation}
where $s\in \bR$ is a parameter. Related deformations have been studied extensively in various contexts by Tian-Zhang \cite{TianZhang}, Ma-Zhang \cite{MaZhangVergneConj}, Braverman \cite{Braverman2002} and many others. 
\begin{lemma}[\cite{LSWittenDef}]
\label{l:wittendef1}
Let $(\E,\nabla^\E,Q)$ be an admissible cycle. The square $\DiracQs_s$ is given by
\[ \DiracQs_s=\DiracQs+s^2\bar{\Phi}^2+s[\DiracE,\bar{\Phi}]\]
where $\bar{\Phi}^2=\nu^2+\bar{\kappa}^2$,
\[ [\DiracE,\bar{\Phi}]=2\nu^2+4\pi(\phi+\phi_\E)\cdot \bar{\phi}+\i [\DiracE,\c_{\t^\perp}(\nu)]-\i\sum_{j=1}^{\dim(\t)} 2\bar{\phi}_j\c(\nabla_\bullet \xi^j_X)+\c(\d \bar{\phi}_j)\c(\xi^j_X)-2\bar{\phi}_j\L_{\xi^j}, \]
and $\xi^1,...,\xi^{\dim(\t)}$ is a basis of $\t$. For each $\lambda \in \Lambda$ there is a proper and bounded below function $f_\lambda$ on $X^\circ$ such that the operator inequality $[\DiracE,\bar{\Phi}]_{[\lambda]}\ge f_\lambda$ holds. For each $s\ge 0$ the operator $(\DiracQs_s)_{[\lambda]}$ is Fredholm, and its index does not depend on $s$.
\end{lemma}
The proof is essentially the same as the proof of a similar result in \cite[Section 4]{LSWittenDef} for the special case where $\E$ is a finite dimensional $T\times \Pi$-equivariant vector bundle and $Q=0$, and so we will not repeat it here. One can take $f_\lambda=f-c_\lambda$ where $c_\lambda$ is a constant only depending on $\lambda$ (its appearance comes from the Lie derivative term in the expression for the cross-term $[\DiracE,\bar{\Phi}]$). The lemma suggests that as $s \rightarrow \infty$, the term that is quadratic in $s$ dominates and the kernel of $\DiracQ_s$ localizes near the subset $Z \subset X$ where $\bar{\Phi}$ vanishes, the fact that $f_\lambda$ is proper and bounded below meaning that the cross-term is under control. Making this statement precise involves some analysis to which the rest of this section is devoted. 

For each $\beta_* \in \B_*$ let 
\[ U'_{\beta_*} \subsetneq U_{\beta_*}\subset \U_{\beta_*} \cap (|\nu|)^{-1}[0,1) \]
be $T$-invariant open neighborhoods of $Z_{\beta_*}$, where $\U_{\beta_*}$ is the neighborhood of $\Z_{\beta_*}$ introduced in Section \ref{s:Kirwanvf}. We may assume the closure $\ol{U}_{\beta_*}$ is a smooth manifold with boundary. For $\beta=\beta_*+\eta \in \B$ where $\eta \in \Pi$, let $U_\beta=\eta\cdot U_{\beta_*}$, $U'_\beta=\eta\cdot U'_{\beta_*}$, and put
\[ U=\bigcup_{\beta \in \B} U_\beta, \qquad U'=\bigcup_{\beta \in \B} U_\beta'.\]
By Proposition \ref{p:Kirwanpos} we may arrange that $U,U'$ have the property that 
\[ \kappa^2(x)+\nu^2(x)>c^2>0 \qquad \text{for} \qquad x \in U\backslash U'. \]
Let $\rho^{1/2}\colon X\rightarrow [0,1]$ be a smooth $T$-invariant function equal to $1$ on $U'$ and satisfying $\rho(U)\subset (0,1]$, $\supp(\rho)=\ol{U}$, $\|\d \rho\|_\infty<\infty$ (the latter property can be arranged by Proposition \ref{p:weakperiodicity}). Define operators
\begin{equation} 
\label{e:DiracQU}
\DiracQ_{U,s}=\DiracQ_U+s\rho^{-1}\Phi, \qquad \DiracQ_U=\rho^{1/2}\DiracE \rho^{1/2}+Q
\end{equation}
where
\begin{equation}
\label{e:Phi}
\Phi=\i(\c_{\t^\perp}(\nu)\hotimes 1-1\hotimes \c(\kappa)).
\end{equation}
Note that $\Phi$ differs from $\bar{\Phi}$ in that $\bar{\kappa}$ has been replaced with $\kappa$. We regard $\DiracQ_U$ as an unbounded operator in the Hilbert space $L^2(U,\E\hotimes S|_U)$ where the measure is the restriction of the Riemannian measure on $X$ (although $U$ is not complete, $\DiracQ_U$ is still essentially self-adjoint thanks to the bump function $\rho^{1/2}$).
\begin{lemma}
\label{l:wittendef2}
Let $(\E,\nabla^\E,Q)$ be an admissible cycle. The square $\DiracQs_{U,s}$ is given by
\[ \DiracQs_{U,s}=\DiracQs_U+s^2\rho^{-2}\Phi^2-s\rho^{-1}\c(\d \rho)\Phi+s[\DiracE,\Phi],\]
where $\Phi^2=\nu^2+\kappa^2$,
\[ [\DiracE,\Phi]=2\nu^2+4\pi\phi^2+4\pi\phi_\E\cdot \phi+\i [\DiracE,\c_{\t^\perp}(\nu)]-\i\sum_{j=1}^{\dim(\t)} 2\phi_j\c(\nabla_\bullet \xi^j_U)+\c(\d\phi_j)\c(\xi^j_U)-2\phi_j\L_{\xi^j}, \]
and $\xi^1,...,\xi^{\dim(\t)}$ is a basis of $\t$. There is a constant $s_0$ such that for $s>s_0$ and for each $\lambda \in \Lambda$ there is a proper and bounded below function $f_{U,\lambda}$ on $U$ such that there is an operator inequality
\[ s\rho^{-2}\Phi^2-\rho^{-1}\c(\d \rho)\Phi+[\DiracE,\Phi]_{[\lambda]}\ge (s-s_0)\rho^{-2}\Phi^2+f_{U,\lambda}.\]
For $s>s_0$ the operator $(\DiracQ_{U,s})_{[\lambda]}$ is Fredholm and its index does not depend on $s$.
\end{lemma}
\begin{proof}
The formula for the square $\DiracQs_{U,s}$ involves a calculation similar to that in Lemma \ref{l:wittendef1}. Consider the expression for the graded commutator $[\DiracE,\Phi]$. The terms involving $[\DiracE,\c_{\t^\perp}(\nu)]$, $\c(\d\phi_j)$ are bounded on $U$ and unimportant. The total Lie derivative $\L_{\xi_j}$ becomes bounded after restricting to the $\lambda$-isotypical component. The dominant term in the expression is therefore $4\pi \phi^2$, as all the other terms are at most linear in $\phi$. It follows that there is a proper and bounded below function $f'_\lambda$ on $X$ such that the operator inequality $[\DiracE,\Phi]_{[\lambda]}\ge f'_\lambda$ holds. Therefore
\begin{equation} 
\label{e:localizedop1}
s\rho^{-2}\Phi^2-\rho^{-1}\c(\d \rho)\Phi+[\DiracE,\Phi]_{[\lambda]}\ge s\rho^{-2}\Phi^2-\rho^{-1}|\d \rho|\cdot|\Phi|+f'_\lambda.
\end{equation}
For any $s_0>0$ we may rearrange the right hand side of \eqref{e:localizedop1} as
\[ (s-s_0)\rho^{-2}\Phi^2+(\tfrac{1}{2}s_0\rho^{-2}\Phi^2+f_\lambda')+(\tfrac{1}{2}s_0\rho^{-1}|\Phi|-|\d \rho|)\rho^{-1}|\Phi|. \]
Note that $\rho^{-1}(x)\ge 1$ for all $x \in U$. On the other hand $\d \rho|_{U'}=0$ and $|\Phi(x)|>c>0$ for $x \in U\backslash U'$. Therefore choosing $\frac{1}{2}s_0>c^{-1}\|\d \rho\|_\infty$ ensures $(\frac{1}{2}s_0\rho^{-1}|\Phi|-|\d \rho|)\ge 0$. Dropping this term \eqref{e:localizedop1} becomes
\[ s\rho^{-2}\Phi^2-\rho^{-1}\c(\d \rho)\Phi+[\DiracE,\Phi]_{[\lambda]}\ge (s-s_0)\rho^{-2}\Phi^2+(\tfrac{1}{2}s_0\rho^{-2}\Phi^2+f'_\lambda).\]
The function $f_{U,\lambda}=\frac{1}{2}s_0\rho^{-2}\Phi^2+f'_\lambda$ is proper and bounded below on $U$, because $\Phi^2>c^2>0$ on $U \backslash U'$, $\rho^{-2}(x)\rightarrow \infty$ as $x \rightarrow \partial U$, and $f'_\lambda$ is proper and bounded below on the closure of $U$. This proves the desired inequality. The remaining statements follow from the inequality.
\end{proof}

\begin{theorem}
\label{t:nonabelloc1}
Let $(\E,\nabla^\E,Q)$ be an admissible cycle. For $s$ sufficiently large $\index_T(\DiracQ)=\index_T(\DiracQ_{U,s})$.
\end{theorem}
\begin{proof}
The difference of the indices equals the index of the block diagonal operator $\DiracQ_s\oplus \DiracQ_{U,s}$ acting on $\H\oplus \H_U^{\op}$, where $\H=L^2(X^\circ,\E\hotimes S)$, $\H_U=L^2(U,\E\hotimes S|_{U})$ and $\H_U^{\op}$ denotes $\H_U$ equipped with the opposite $\bZ_2$-grading. Let $\chi \colon X \rightarrow [0,1]$ be a smooth $T$-invariant bump function equal to $1$ on a neighborhood of $Z$ and such that $\supp(\chi)\subset \rho^{-1}(1)$ (hence $\rho \chi=\chi$) and $\|\d \chi\|_{\infty}<\infty$ (possible because of Proposition \ref{p:weakperiodicity}). Let $m_\chi$ denote multiplication by $\chi$ followed by restriction to $U$, viewed as an odd operator $\H \rightarrow \H_U^{\tn{op}}$. Its adjoint $m_\chi^*\colon \H_U^{\tn{op}}\rightarrow \H$ is multiplication by $\chi|_U$ followed by extension to $X^\circ$ by $0$. Let $\gamma$ be the grading operator for $\H$. Define an odd unbounded self-adjoint operator on $\H\oplus \H_U^{\op}$:
\[ D_s=\matr{\DiracQ_s}{s\gamma m_\chi}{sm_\chi^* \gamma}{\DiracQ_{U,s}}. \]
Fix $\lambda \in \Lambda$. The operator $(D_s)_{[\lambda]}$ is a bounded perturbation of the compact resolvent operator $(\DiracQ_s\oplus \DiracQ_{U,s})_{[\lambda]}$ and therefore has the same index. We claim that for $s \gg 0$, $(D_s)_{[\lambda]}$ does not contain $0$ in its spectrum and hence has index $0$. In fact we prove the stronger result that $(1+D_s^2)_{[\lambda]}^{-1} \rightarrow 0$ in norm as $s \rightarrow \infty$.

Using $[\DiracQ,\chi]=\c(\d\chi)$, $\gamma \DiracQ=-\DiracQ \gamma$ and $\chi \rho=\chi$, one finds
\begin{equation} 
\label{e:2by2}
1+(D_s^2)_{[\lambda]}=\matr{1+(\DiracQs_s)_{[\lambda]}+s^2\chi^2}{-s\gamma \c(\d\chi)}{s\c(\d\chi)\gamma}{1+(\DiracQs_{U,s})_{[\lambda]}+s^2\chi^2}=:\matr{w_s}{x_s}{y_s}{z_s}. 
\end{equation}
Note that $w_s$, $z_s$ are invertible for all $s$. By Lemma \ref{l:wittendef1},
\begin{equation} 
\label{e:ws}
s^{-1}w_s\ge s(\Phi^2+\chi^2)+f_{\lambda},
\end{equation}
for a function $f_{\lambda}$ that is proper and bounded below. Since $\Phi^2+\chi^2>0$ everywhere and $f_{\lambda}$ is proper and bounded below, the infimum of $s(\Phi^2+\chi^2)+f_{\lambda}$ goes to infinity as $s \rightarrow \infty$. Hence by taking inverses, \eqref{e:ws} shows that $\|w_s^{-1}\|=o(s^{-1})$ as $s \rightarrow \infty$. By Lemma \ref{l:wittendef2},
\begin{equation}
\label{e:ws2}
s^{-1}z_s\ge (s-2s_0)(\rho^{-2}\Phi^2+\chi^2)+f_{U,\lambda}
\end{equation}
for a function $f_{U,\lambda}$ that is proper and bounded below on $U$. Since $\Phi^2+\chi^2>0$ everywhere and $f_{U,\lambda}$ is proper and bounded below on $U$, the infimum of $(s-2s_0)(\rho^{-2}\Phi^2+\chi^2)+f_{U,\lambda}$ goes to infinity as $s \rightarrow \infty$. Hence by taking inverses, \eqref{e:ws2} shows that $\|z_s^{-1}\|=o(s^{-1})$ as $s \rightarrow \infty$. On the other hand $\|x_s\|$, $\|y_s\|$ are both $O(s)$ because $\|\d \chi\|_\infty<\infty$.

Define
\[ q_s=w_s-x_sz_s^{-1}y_s=w_s(1-w_s^{-1}x_sz_s^{-1}y_s).\]
By the order estimates above, $\|w_s^{-1}x_sz_s^{-1}y_s\|\rightarrow 0$ as $s \rightarrow \infty$. Hence for $s\gg 0$, $q_s$ is invertible and moreover $\|q_s^{-1}\|=o(s^{-1})$. This in turn means we can apply the following explicit formula for the inverse of a block $2\times 2$ matrix:
\begin{equation} 
\label{e:inverse2by2}
(1+D_s^2)^{-1}_{[\lambda]}=\matr{w_s}{x_s}{y_s}{z_s}^{-1}=\matr{q_s^{-1}}{q_s^{-1}x_sz_s^{-1}}{z_s^{-1}y_sq_s^{-1}}{z_s^{-1}+z_s^{-1}y_sq_s^{-1}x_sz_s^{-1}}.
\end{equation}
By the order estimates above, each entry converges to $0$ in norm as $s \rightarrow \infty$.
\end{proof}
\begin{remark}
An analytic localization argument employing the explicit formula \eqref{e:inverse2by2} for the inverse of a block $2\times 2$ matrix appeared in the work of Bismut and Lebeau \cite[Chapter IX]{BismutLebeau}.
\end{remark}

Since $U=\sqcup_{\beta \in \B} U_\beta$, the theorem yields an expression for $\index_T(\DiracQ)$ as a sum of contributions from components $Z_\beta$, $\beta \in \B$ of $Z$.
\begin{corollary}
\label{c:analyticnonabelian}
Let $(\E,\nabla^\E,Q)$ be an admissible cycle. For each $\beta \in \B$, let $\DiracQ_{U_\beta,s}$ be the restriction of the operator $\DiracQ_{U,s}$ to $U_\beta \subset U$. For $s\gg 0$ the equation
\[ \index_T(\DiracQ)=\sum_{\beta \in \B}\index_T(\DiracQ_{U_\beta,s}) \]
holds in $R^{-\infty}(T)$.
\end{corollary}

\subsection{Non-abelian localization and transversely elliptic symbols}\label{s:nonabellocII}
The contributions $\index_T(\DiracQ_{U_\beta,s})$, $\beta \in \B$ to the index formula proved in Corollary \ref{c:analyticnonabelian} are conveniently described in terms of indices of transversely elliptic operators \cite{AtiyahTransEll}.
\begin{definition}
\label{d:symbol}
Let $\sigma \colon T^*X^\circ\rightarrow \End(S)$ be the product symbol
\[ \sigma(x,p)=1 \hotimes \i\c(p)+\i \c_{\t^\perp}(\nu(x))\hotimes 1, \qquad (x,p) \in T^*X.\]
Using the Kirwan vector field, define the deformed symbol
\[ \sigma_\kappa(x,p)=1 \hotimes \i\c(p-\kappa(x))+\i \c_{\t^\perp}(\nu(x))\hotimes 1, \qquad (x,p) \in T^*X^\circ\]
and let $\sigma_\beta=\sigma_{\kappa}\upharpoonright T^*U_\beta$.
\end{definition}
If $H$ is a compact Lie group acting on a manifold $M$, one says that an $H$-invariant symbol $\sigma$ is $H$-\emph{transversely elliptic} (\cite{AtiyahTransEll}) if $\sigma\upharpoonright T^*_H M$ is invertible outside a compact subset, where $T^*_H M \subset T^*M$ is the conormal space to the $H$ orbits.
\begin{proposition}
The symbol $\sigma_\beta$ is $T$-transversely elliptic.
\end{proposition}
\begin{proof}
Use $g$ to identify $TX^\circ=T^*X^\circ$. Let $\Gamma(\kappa)\subset TX^\circ$ be the graph of $\kappa$ (restricted to $X^\circ$), and let $\pi\colon T^*X^\circ \rightarrow X^\circ$ be the projection. The subset of $T^*U_\beta$ where $\sigma_\beta$ fails to be invertible is $T^*U_\beta\cap \Gamma(\kappa)\cap \pi^{-1}(\nu^{-1}(0))$. As $\kappa$ is tangent to the group orbit directions, the intersection of $\Gamma(\kappa)$ with $T^*_TX^\circ$ is the vanishing locus of $\kappa$, viewed as a subset of the $0$ section. Thus
\[ T^*_TU_\beta \cap \Gamma(\kappa)\cap \pi^{-1}(\nu^{-1}(0))=Z_\beta \]
and is compact.
\end{proof}
Invertibility outside a compact implies that an $H$-transversely elliptic symbol $\sigma$ defines an element in the \emph{compactly supported} K-theory group $K_H^c(T^*_HM)$; we denote this K-theory class by $\sigma$ as well when there is no risk of confusion. On a compact $H$-manifold Atiyah \cite{AtiyahTransEll} proved that the $H$-equivariant index of an $H$-transversely elliptic operator makes sense as a distribution on $H$, and depends only on the K-theory class defined by its symbol. Based on this Atiyah defines the \emph{analytic index} of any class $\sigma \in K^c_H(T^*_HM)$, denoted $\index_H(\sigma) \in R^{-\infty}(H)$, as the index of any transversely elliptic operator with symbol $\sigma$. More generally, on a non-compact manifold (such as $U_\beta$) Atiyah defines the index by embedding in a compact manifold and extending the symbol appropriately. 

Recall that $K^c_H(T^*_HM)$ is a $K_H(M)$-module, and we denote the module action with an $\otimes$ symbol. Given $\sigma \in K^c_H(T^*_HM)$ and $\sf{E}\in K_H(M)$, represent $\sf{E}$ by a difference of vector bundles $E_1,E_2$ near the projection to $M$ of the \emph{compact} subset where $\sigma$ fails to be invertible, and extend $E_1,E_2$ arbitrarily to $M$. Then $\sigma \otimes \sf{E}$ is the difference of the K-theory classes represented by the symbols $\sigma \otimes \id_{E_1}$, $\sigma\otimes \id_{E_2}$.

\begin{theorem}
\label{t:nonabelloc}
Let $\sf{E}\in K^{\ad}_T(\M)$ be an admissible class whose pullback to $X$ is represented by the admissible cycle $(\E,\nabla^\E,Q)$. Let $\sf{E}_\beta \in K_T(U_\beta)$ be the pullback of $\sf{E}$ to $U_\beta$. Then
\begin{equation} 
\label{e:Kthynonabelloc}
\index_T(\DiracQ)=\sum_{\beta \in \B} \index_T(\sigma_\beta \otimes \sf{E}_\beta).  
\end{equation}
In particular the index depends only on the K-theory class $\sf{E}$.
\end{theorem}
\begin{proof}
By Corollary \ref{c:analyticnonabelian} one needs to show $\index_T(\DiracQ_{U_\beta,s})=\index_T(\sigma_\beta \otimes \sf{E}_\beta)$. Recall $U_\beta$ is a small neighborhood of the compact set $Z_\beta=X^\beta \cap \phi^{-1}(\beta)\cap \nu^{-1}(0)$ where the $0$-th order deformation $\rho^{-1}\Phi|_{U_\beta}=\i \rho^{-1}(\c_{\t^\perp}(\nu)\hotimes 1-1\hotimes \c(\kappa))|_{U_\beta}$ vanishes. Since $Z_\beta$ is compact, the result follows from known results relating $0$-th order deformations of this type to indices of transversely elliptic symbols: details and various approaches may be found in \cite{Braverman2002, MaZhangBravermanIndex, LSWittenDef, LRSkkbrav}.
\end{proof}
In view of Theorem \ref{t:nonabelloc}, we can make the following definition.
\begin{definition}
\label{d:analyticindex}
The \emph{equivariant analytic index} map is the $R(T)$-module map
\[ \index_T \colon K^\ad_T(\M)\rightarrow R^{-\infty}(T), \qquad \sf{E}\mapsto \index_T(\DiracQ),\]
where $(\E,\nabla^\E,Q)$ is any admissible cycle representing the pullback of $\sf{E}$ to $X$.
\end{definition}

In Theorem \ref{t:nonabelloc}, the contribution labelled by $\beta \in \B$ can be expressed in terms of objects defined on the fixed-point set $(U_\beta)^\beta$. The element $\beta$ determines a complex structure on the normal bundle $\nu(\beta)$ to $(U_\beta)^\beta$ with the property that the eigenvalues of the action of $\i^{-1}\beta$ on $\nu(\beta)^{1,0}$ are positive. There is a Clifford module $S(\beta)$ on $(U_\beta)^\beta$ with the property 
\[ S(\beta) \hotimes \wedge \nu(\beta)^{0,1}\simeq S|_{(U_\beta)^\beta},\] 
and a corresponding induced transversely elliptic symbol $\sigma_{S(\beta)}$. Then
\begin{equation}
\label{e:NonAbSym}
\index_T(\sigma_\beta \otimes \sf{E}_\beta)=\index_T(\sigma_{S(\beta)}\otimes \Sym(\nu(\beta)^{1,0})\otimes \sf{E}_\beta),
\end{equation}
where $\Sym(\nu(\beta)^{1,0})$ denotes the symmetric algebra bundle. See \cite{ParadanRiemannRoch, WittenNonAbelian} and \cite[Theorem 6.14]{LSWittenDef} for further details.

The local finiteness of the sum (in $\bZ^\Lambda=R^{-\infty}(T)$) in Theorem \ref{t:nonabelloc} can be verified directly using \eqref{e:NonAbSym}. Because of the $\Sym(\nu(\beta)^{1,0})$ factor, the contribution of $0\ne \beta \in \B$ has support contained in a half space of the form $H(\beta,c_\beta)=\{\xi \in \t^*|\xi\cdot \beta/|\beta|\ge c_\beta\}$. We claim that the constants $c_\beta \rightarrow \infty$ as $|\beta|\rightarrow \infty$, from which the result follows. The constant $c_\beta$ depends on the eigenvalues for the action of $\i^{-1}\beta/|\beta|$ on $S(\beta)$ and $\E|_{Z_\beta}$. If $\E=\bC$, then the claim is true and follows from our assumption $\ell>0$, equation \eqref{e:commutation}, and Proposition \ref{p:weakperiodicity}, which imply that the eigenvalues for the action of $\i^{-1}\beta/|\beta|$ on $S(\beta)$ go to $+\infty$ as $|\beta|\rightarrow \infty$; see for example \cite{LSWittenDef} for further details. In the general case the result is still true because of $T$-finiteness: the eigenvalues for the action of $\beta/|\beta|$ on $\E|_{Z_\beta}$ are bounded by a constant independent of $\beta$.

\subsection{Weyl symmetry}
Definition \ref{d:admissible} has an obvious $N(T)$-equivariant analogue where we require $(\E,\nabla^\E)$ to be $N(T)\ltimes \Pi$-equivariant and $Q$ to be $N(T)$-equivariant; we take this as the definition of the admissible classes $K^\ad_{N(T)}(X)\subset K_{N(T)}(X)$. Let $K_G^\ad(\M)$ be the $R(G)$-subalgebra of classes whose pullback to $X$ lies in $K^\ad_{N(T)}(X)$. The index of a class in $K_G^\ad(\M)$ has an additional antisymmetry under the action of the Weyl group.
\begin{proposition} 
\label{p:Weylanti}
If $\sf{E} \in K^\ad_G(\M)$ then $\index_T(\sf{E})$ is $W$-antisymmetric.
\end{proposition}
\begin{proof}
The Dirac operator $\Dirac_{\scr{S}}$ is $N(T)$-invariant. The operator $\Dirac$ represents the KK-product of a K-homology class defined by $\Dirac_{\scr{S}}$ and the pullback of the $T$-equivariant Bott-Thom element $\scr{B}$ for $\t^\perp$. As mentioned in Section \ref{s:DiracTrans}, the latter element is $W$-antisymmetric, and the result follows from this; see \cite[Section 4.5]{LSQuantLG} for further details.
\end{proof}
Let $J=\sum_{w \in W}(-1)^{l(w)}e^{w\rho}$ be the Weyl denominator and let $V_\lambda$ be the irreducible representation of $G$ with highest weight $\lambda \in \Lambda_+$. There is an isomorphism $R^{-\infty}(G)\xrightarrow{\sim} R^{-\infty}(T)^{W-\tn{anti}}$ given by
\[ \sum_{\lambda \in \Lambda_+}n_\lambda \Tr_{V_\lambda}\mapsto \sum_{\lambda \in \Lambda_+}n_\lambda J\cdot \Tr_{V_\lambda}|_T=\sum_{\lambda \in \Lambda_+}c_\lambda \sum_{w \in W}(-1)^{l(w)}e^{w(\lambda+\rho)}. \]
Let $I^G_T$ be the inverse map. Then $I^G_T$ sends a Weyl numerator to the corresponding irreducible character. Equivalently $I^G_T$ is $|W|^{-1}$ times the Dirac induction map. Formally, we define a $G$-equivariant index map
\[ \index_G=I^G_T\circ \index_T \colon K^\ad_G(\M)\rightarrow R^{-\infty}(G).\]
In this case the non-abelian localization formula can be written in a slightly more symmetrical form,
\[ \index_G(\sf{E})=\sum_{\beta \in \B_+} I^G_T\index_T\big(\sum_{w \in W}\sigma_{w\beta} \otimes \sf{E}_{w\beta}\big), \]
where $\B_+=\B\cap \t_+$.

\section{Cohomological formulas}\label{s:cohomology}
In this section we derive two cohomological formulas for the index of an admissible class $\sf{E} \in K^\ad_T(\M)$ under additional hypotheses on the existence of a well-behaved equivariant Chern character form for $\sf{E}$ along $X$. Using a cohomological index formula for transversely elliptic operators due to Paradan-Vergne \cite{ParVerTransEll} (building on work of Berline-Vergne \cite{BerlineVergneTransversallyElliptic}) and a non-abelian localization formula in cohomology, we obtain a Kirillov-type formula for the index. Under an additional twisted equivariance assumption on the Chern character form we obtain an abelian localization formula involving integration over a compact manifold.

\subsection{Index formula for transversely elliptic operators}
A cohomological index formula for transversely elliptic operators due to Paradan and Vergne \cite{ParVerTransEll} can be applied to each term on the right hand side of Theorem \ref{t:nonabelloc}. The result is expressed in terms of $T$-equivariant characteristic forms: these are differential forms $\alpha(\xi)$ on $M$ depending smoothly on a parameter $\xi \in \t$ (sometimes only defined on a neighborhood of $0 \in \t$), defined by replacing the curvature with the equivariant curvature in the formula for the Chern-Weil representative, cf. \cite{BerlineGetzlerVergne, MeinrenkenEncyclopedia, ParVerTransEll}. For example the equivariant Todd form for a $T$-equivariant vector bundle with connection $(E,\nabla^E)$ is
\[ \Td(E,\xi)=\tn{det}_\bC\bigg(\frac{(\i/2\pi)F^E(\xi)}{1-\exp(-(\i/2\pi)F^E(\xi))}\bigg), \quad F^E(\xi)=(\nabla^E)^2-2\pi \i(\L^E_\xi-\nabla^E_\xi), \]
and is $\d_\xi$-closed where $\d_\xi=\d+2\pi \i\iota(\xi_X)$. 

The class of the symbol $\sigma$ in K-theory is the product of the class defined by the symbol $\sigma_{\scr{S}}$ of the Dirac operator for $\scr{S}$, and of the pullback under $\nu \colon X^\circ \rightarrow \t^\perp$ of the Bott-Thom element $\scr{B}$ for $0\in \t^\perp$ (see Definition \ref{d:symbol}). Consequently the equivariant Chern character of $\sigma$ is a product
\begin{equation}
\label{e:eqChsigma}
\Ch^u(\sigma,\xi)=\Ch^u(\sigma_{\scr{S}},\xi)\nu^*\Ch^u(\scr{B},\xi).
\end{equation} 
The Bott-Thom element $\scr{B}$ has equivariant Chern character
\begin{equation}
\label{e:ChBott}
\Ch^u(\scr{B},\xi)=(u\exp(\xi))^{-\rho}\tn{det}^{\n/\n^u}(1-u\exp(\xi))\tn{det}^{\n^u}\Big(\frac{1-\exp(\xi)}{\xi}\Big)\scr{T}_{(\t^\perp)^u}(\xi)
\end{equation}
where $\scr{T}_{(\t^\perp)^u}(\xi)\in \Omega_c((\t^\perp)^u)$ is an equivariant Thom form (whose pullback to $0 \in (\t^\perp)^u$ is $\tn{det}^{\n^u}_\bC(\xi)$). Hence one can arrange that the support of $\Ch^u(\sigma,\xi)$ is compact in the vertical direction in $T^*X^u$, and lies over a $\Pi$-invariant subset $X'\subset X^\circ\subset X$ such that $\nu(X')$ is a compact subset of $\t^\perp$ (and the latter subset of $\t^\perp$ can be chosen as small as one wishes by choosing $\scr{T}_{(\t^\perp)^u}(\xi)$ appropriately); this also has the consequence that $X'/\Pi$ is a compact subset of $X^\circ/\Pi$.

Another important ingredient is the equivariant differential form with generalized coefficients $P_\beta(\xi)$ (i.e. a generalized function of $\xi$ taking values in smooth differential forms on $X$---see \cite{KumarVergne} for general background) introduced by Paradan \cite{Paradan97,Paradan98}, given by the expression
\begin{equation} 
\label{e:Pbeta}
P_\beta(\xi)=\chi_\beta-\d\chi_\beta \cdot \Theta \int_0^\infty e^{s\d_\xi\Theta}\d s, \qquad \Theta=g(\kappa)\in \Omega^1(X)^T
\end{equation}
where $\chi_\beta$ is a $T$-invariant bump function equal to $1$ on a small neighborhood of $\Z_\beta=X^\beta \cap \phi^{-1}(\beta)$ (see \cite{Paradan98,ParVerTransEll} or the proof of Theorem \ref{t:deloc} below for further details). The oscillatory integral in equation \eqref{e:Pbeta} can be thought of as a replacement for the expression ``$-(\d_\xi \Theta)^{-1}$''. By construction the product $P_\beta(\xi)\Ch^u(\sigma,\xi)$ has \emph{compact} support near $Z_\beta^u \subset T^*U_\beta^u$.

The Paradan-Vergne formula \cite[Theorem 5.6]{ParVerTransEll} applied to $\sigma_\beta \otimes \sf{E}_\beta$ yields that for $u \in T$ and $\xi \in \t$ sufficiently small one has
\begin{equation} 
\label{e:cohomnonabel}
\index_T(\sigma_\beta \otimes \sf{E}_\beta)(u\exp(\xi))=\int_{U_\beta^u} P_\beta(\xi)\A\S^u(\sigma,\xi)\Ch^u(\sf{E}_\beta,\xi),
\end{equation}
where 
\begin{equation}
\label{e:AtiyahSinger}
\A\S^u(\sigma,\xi)=(-1)^{\dim(X^u)/2}\pi_\ast \frac{\Td(T_\bC X^u,\xi)\Ch^u(\sigma,\xi)}{\Ch^u(\bm{\lambda}_{-1}\nu_\bC(X,X^u),\xi)}
\end{equation}
is the integral over the fibres $\pi\colon T^*X^u \rightarrow X^u$ of the equivariant Atiyah-Segal-Singer integrand. The integral on the right hand side of equation \eqref{e:cohomnonabel} converges in the sense of generalized functions, i.e. the integrand should be smeared with a test function supported sufficiently near $0 \in \t$ before the integral over $U_\beta^u$ is carried out.

To make the Atiyah-Singer integrand more concrete, note that in case $u=1$, carrying out the integral over the fibres leads to
\begin{equation} 
\label{e:equivAS}
\A\S(\sigma,\xi)=\Ahat(X,\xi)\Ch(\scr{L},\xi)^{1/2}\nu^*\Ch(\scr{B},\xi) 
\end{equation}
where $\scr{L}$ is the anti-canonical line bundle (Definition \ref{d:spincmoment}). Moreover
\begin{equation}
\label{e:Chhalf}
\Ch(\scr{L},\xi)^{1/2}=e^{\varpi+2\pi \i \pair{\phi}{\xi}},
\end{equation}
where $\varpi$, $\phi$ were introduced in Definition \ref{d:spincmoment}. A similar formula exists for general $u \in T$, although it involves a phase factor that is somewhat cumbersome to describe (cf. \cite{BerlineGetzlerVergne, duistermaatheat, AMWVerlinde, YiannisThesis}). 

Equation \eqref{e:cohomnonabel} holds for $\xi$ contained in a ball of some radius $r_{\beta,u}>0$ in $\t$. The maximal $r_{\beta,u}$ depends on the geometry of the $T$-action near $Z_\beta$, and in general it can happen that $\{r_{\beta,u}|\beta \in \B\}$ has no positive lower bound. An important feature of our setting is that Proposition \ref{p:weakperiodicity} guarantees that this does not occur, i.e. one can find a uniform radius $r_u>0$ such that \eqref{e:cohomnonabel} holds for $\xi$ in a ball $B_{r_u}$ of radius $r_u$ in $\t$ for all $\beta \in \B$. 
Since $T$ is compact, we may find a finite set $u_1,...,u_N$ such that the sets $u_i\exp(B_{r_{u_i}})$ cover $T$. Let $\rho_i \in C^\infty(T)$, $i=1,...,N$ be a partition of unity subordinate to the cover. By \eqref{e:cohomnonabel},
\begin{equation} 
\label{e:cohomnonabel2}
\index_T(\sigma_\beta \otimes \sf{E}_\beta)(u)=\sum_{i=1}^N\rho_i(u)\int_{U_\beta^{u_i}} P_\beta(\xi_i)\A\S^{u_i}(\sigma,\xi_i)\Ch^{u_i}(\sf{E}_\beta,\xi_i),
\end{equation}
where $\xi_i=\log(u_i^{-1}u)$.

Our next aim is to replace the contributions \eqref{e:cohomnonabel} localized around the components $X^u\cap Z_\beta$ with a global expression involving integration over $X^u$. In brief the strategy (which we learned from \cite{ParVerSemiclassical}) is to use the non-abelian localization formula in cohomology (\cite{Paradan98}) in reverse. Since $X^u$ is non-compact and the integral over $X^u$ will only converge in the sense of generalized functions, the argument requires some care in choosing well-behaved differential form representatives.

We have already written the Atiyah-Segal-Singer integrand \eqref{e:AtiyahSinger}, \eqref{e:equivAS} in a way that suggests it makes sense globally on $X^u$. Recall that the Spin$_c$ structure $\scr{S}$ is $N(T)\ltimes \wh{\Pi}^{(\ell)}$-equivariant. Using invariant connections in the equivariant Chern-Weil construction yields a differential form representing $\A\S^u(\sigma,\xi)$ that is $\Pi$-periodic up to a phase factor:
\begin{equation} 
\label{e:ASPiEquiv}
\eta^*\A\S^u(\sigma,\xi)=u^{\ell \eta}e^{2\pi \i\ell \eta\cdot\xi}\A\S^u(\sigma,\xi) 
\end{equation}
for $\eta \in \Pi$; for example, in the case $u=1$, this is clear from \eqref{e:equivAS}: $\Ahat(X,\xi)\Ch(\scr{B},\xi)$ is $\Pi$-periodic, while $\Ch(\scr{L},\xi)^{1/2}$ changes by the phase factor $e^{2\pi \i\ell \eta\cdot\xi}$ due to \eqref{e:Piequivphi}.

\subsection{Chern character forms}
The next definition lists desirable properties of a globally defined differential form on $X$ whose pullback to each relatively compact subset $U_\beta$ represents the Chern character of $\sf{E}_\beta$.
\begin{definition}
\label{d:uniformChern}
Let $\sf{E} \in K^\ad_T(\M)$ be an admissible class. Let $\Ch^u(\sf{E},\xi)$, $u \in T$ denote any collection of smooth $T$-equivariant forms on $X^u$, depending smoothly on $\xi$, closed for the differential $\d+2\pi\i\iota(\xi_X)$, and having the following properties:
\begin{enumerate}
\item As $u$ varies, the family forms a \emph{bouquet} in the sense of \cite[Definition 60]{DufloVergneDescent}, \cite[p. 150]{vergne1996multiplicities}. In the abelian case this simply means: for $t \in T$ and $\xi \in \t$ sufficiently small, the pullback of $\Ch^u(\sf{E},\xi+\xi')$ to $X^{u\exp(\xi)}$ is $\Ch^{u\exp(\xi)}(\sf{E},\xi')$.
\item The $T$-equivariant cohomology class of the pullback of $\Ch^u(\sf{E},\xi)$ to $U_\beta^u$ is $\Ch^u(\sf{E}_\beta,\xi)$.
\item There is a compact subset $K \subset \t^*$ such that for all $u \in T$ and $x \in X^u$, the Fourier transform $\scr{F}_\t\Ch^u(\sf{E},\cdot)_x$ of the smooth function
\[ \xi \in \t \mapsto \Ch^u(\sf{E},\xi)_x \in \wedge T^*_xX_\bC \]
has support contained in $K$.
\item There are integers $m,m'\ge 0$ and a constant $C>0$ such that for all $u \in T$, $\xi \in \t$ and $x \in X^u$,
\begin{equation} 
\label{e:chernestimate}
|\Ch^u(\sf{E},\xi)_x|\le C(1+|\phi(x)|^2)^{m'/2}(1+|\xi|^2)^{m/2}.
\end{equation}
\end{enumerate}
\end{definition}
\begin{remark}
The assumption (d) could be expressed in terms of the Fourier transform: there exist $C,m,m'$ such that
\begin{equation} 
\label{e:chernestimatefourier}
\pair{\scr{F}_\t\Ch^u(\sf{E},\cdot)_x}{f}\le C(1+|\phi(x)|^2)^{m'/2}(\|f\|_\infty+\|\nabla f\|_\infty+\cdots+\|\nabla^mf\|_\infty),
\end{equation}
for all $f \in C^\infty_c(\t)$.
\end{remark}
\begin{example}
\label{ex:vbexample}
Suppose the pullback of $\sf{E}$ to $X$ can be represented by a smooth $T$-equivariant Hermitian vector bundle $E \rightarrow X$ of rank $r$ with compatible connection $\nabla^E$ having bounded curvature $\|F_E\|_\infty<C_1$ and such that the moment map determined by the connection $\phi_E \in \t^* \otimes C^\infty(X,\End(E))$ satisfies $\|\phi_E\cdot \xi\|_\infty<C_2|\xi|$. In particular any $T\times \Pi$-equivariant vector bundle has such a connection. Then the corresponding equivariant Chern-Weil form on $X^u$ (cf. \cite[Chapter 7]{BerlineGetzlerVergne}, \cite{MeinrenkenEncyclopedia}),
\[ \Tr_E(ue^{(\i/2\pi)F_E(\xi)}), \]
where
\[ \frac{\i}{2\pi}F_E(\xi)=\frac{\i}{2\pi}F_E+2\pi \i\phi_E\cdot\xi, \quad F_E=(\nabla^E)^2,\quad 2\pi\i\phi_E\cdot \xi=\L^E_\xi-\nabla^E_{\xi_X},\]
satisfies the conditions in Definition \ref{d:uniformChern}. Conditions (a),(b) are immediate. Conditions (c),(d) are straight-forward consequences of the bounds $\|F_E\|_\infty<C_1$, $\|\phi_E\cdot \xi\|_\infty<C_2|\xi|$. Indeed to keep the notation simple, consider the case where the rank of $E$ is $1$ (the general case is similar). Then $\phi_E$ can be viewed as a smooth $\t^*$-valued function with norm bounded by $C_2$. The Fourier transform of $\Tr_E(ue^{(\i/2\pi)F_E(\xi)})$ at the point $x \in X^u$ is the generalized function
\[ \nu\mapsto u_xe^{(\i/2\pi)F_{E,x}}\int_{\t} e^{-2\pi\i\pair{\nu-\phi_E(x)}{\xi}}\d\xi\]
which is $u_xe^{(\i/2\pi)F_{E,x}} \in \wedge T^*_x X_\bC$ times the delta distribution at $\phi_E(x) \in \t^*$. Thus we can take $K$ to be the closed ball in $\t^*$ of radius $C_2$, $m=m'=0$, and $C=C_1^{\dim(X)/2}$.
\end{example}

Under suitable analytic hypotheses Example \ref{ex:vbexample} could be generalized to construct a Chern character form for an admissible cycle $(\E,\nabla^\E,Q)$ via the equivariant Chern-Weil construction for the superconnection built from $\nabla^\E$ and $Q$ (\cite{BerlineGetzlerVergne}). We will not pursue this here. In any case for the Atiyah-Bott classes considered in a sequel article, it is more convenient to simply write down a candidate differential form (suggested by the families index theorem) and verify conditions (a)--(d) of Definition \ref{d:uniformChern} by other means.

\begin{remark}
Existence of a Chern character form for an admissible cycle $(\E,\nabla^\E,\phi_\E)$ satisfying the conditions in Definition \ref{d:uniformChern} is not automatic. For a somewhat artificial example, if $\M=LG\cdot \xi$ is the coadjoint orbit of $\xi \in \t \subset L\g^*$, then the space $X$ is a small thickening of a discrete space $W_\aff\cdot \xi$. Let $\E=X\times \ell^2(\bZ)$ where the $T$ action is trivial, the $\Pi$-action is the obvious one, and equip $\E$ with the trivial connection $\nabla^\E$. Choose operators $Q_{w\cdot \xi}$ on $\E_{w\cdot \xi}$, one for each connected component of $X$, that have compact resolvent, but such that $\index(Q_{w\cdot \xi})$ grows exponentially as $|w\cdot \xi|\rightarrow \infty$. Then $(\E,\nabla^\E,Q)$ is admissible in the sense of Definition \ref{d:admissible}, but there is no Chern character form satisfying \eqref{e:chernestimate}. There are similar examples where $T$ acts non-trivially on the fibres and the non-equivariant index of the operators $Q_{w\cdot \xi}$ is constant, but the multiplicity of some representation in the equivariant index grows exponentially.
\end{remark}

\subsection{Kirillov-type formula}
\ignore{
Hence one can take $\X$ $\chi$ the sum over $\beta \in \B$, which yields
\begin{equation} 
\label{e:sumbeta}
\index_T^\Dirac(\sf{E}_u)(t\exp(\xi))=\sum_{\beta \in \B}\int_{X_\beta^t} P_\beta(\xi)\A\S^t(\sigma,\xi)\Ch^u(\sf{E}_\beta,\xi) 
\end{equation}
valid for $\xi \in \t$ sufficiently small, with the infinite sum converging in the sense of distributions.}

The next result is a Kirillov-type (cf. \cite{BerlineVergneKirillovFormula}) formula for the index of an admissible K-theory class. 
\begin{theorem}
\label{t:deloc}
Let $\sf{E} \in K^\ad_T(\M)$ admit $T$-equivariant Chern character forms $\Ch^u(\sf{E},\xi) \in \Omega(X^u)$ satisfying the conditions in Definition \ref{d:uniformChern}. Then the non-abelian localization formula \eqref{e:Kthynonabelloc} converges in the sense of distributions to a distribution on $T$, and for $u \in T$ and $\xi \in \t$ sufficiently small
\begin{equation}
\label{e:deloc0}
\index_T(\sf{E})(u\exp(\xi))=\int_{X^u} \Ch^u(\sf{E},\xi)\A\S^u(\sigma,\xi).
\end{equation}
\end{theorem}
The integral in \eqref{e:deloc0} is meant in the distributional sense, i.e. the integrand should be paired with a test function supported sufficiently near $0 \in \t$ before the integral over $X^t$ is performed. Moreover the integration can be taken instead over the manifold without boundary $(X^\circ)^u$, since the support of the integrand is contained in a subset $(X')^u \subset (X^\circ)^u$ (which moreover has the property that $X'/\Pi$ is a compact subset of $X^\circ/\Pi$).
\begin{proof}
Let $f \in C^\infty_c(\t)$ with support contained in the region where $\A\S^t(\sigma,\xi)$ is defined. Recall $\Ch(\scr{L},\xi)^{1/2}=e^{\varpi+2\pi \i \phi\cdot \xi}$ and that $\phi(\eta\cdot x)=\phi(x)+\ell \eta$ for all $\eta \in \Pi$. Let
\[ \alpha(\xi)=\Ch^u(\sf{E},\xi)\A\S^u(\sigma,\xi)e^{-\varpi-2\pi \i \phi\cdot \xi}f(\xi).\]
The pairing $\pair{\Ch^u(\sf{E},\cdot)\A\S^u(\sigma,\cdot)}{f}$ of the integrand of \eqref{e:deloc0} with $f$, evaluated at $x \in X$ is
\begin{equation} 
\label{e:pairingwithf}
\int_{\t} \alpha_x(\xi)e^{\varpi_x+2\pi \i\phi_x\cdot\xi}\d\xi=e^{\varpi_x}(\scr{F}_\t\alpha_x)(-\phi_x)\in \wedge T^*_xX. 
\end{equation}
The Fourier transform of a product is the convolution of the Fourier transforms. The function $\gamma_x(\xi)=\A\S^u(\sigma,\xi)_x e^{-\varpi_x-2\pi \i \phi_x\cdot \xi}f(\xi) \in \wedge T^*_xX^\circ$ is smooth and compactly supported (since $f$ is), hence its Fourier transform $\scr{F}_\t\gamma_x$ is a Schwartz function. Under pullback by an element of the lattice $\eta \in \Pi$, $\A\S^u(\sigma,\xi) e^{-\omega-2\pi \i \phi\cdot \xi}$ transforms with a phase $u^{-\ell \eta}$ that does not depend on $\xi$. Since the support of $\A\S^u(\sigma,\xi)$ is contained in the inverse image of a compact subset $X'/\Pi$ of $X^\circ/\Pi$, it follows that the Schwartz semi-norms of $\scr{F}_\t\gamma_x$ are uniformly bounded as functions of $x$. The properties of $\Ch^u(\sf{E},\xi)$ (in particular equation \eqref{e:chernestimatefourier}) imply that the Schwartz semi-norms of the convolution $\scr{F}_\t\alpha_x$ are bounded (as a function of $x$) by a constant times $(1+|\phi_x|^2)^{m'/2}$. In particular the pointwise norm $|\scr{F}_\t\alpha_x(\nu)|$ (taken in $\wedge T_x^*X^u$) satisfies an estimate
\begin{equation} 
\label{e:dombyrapid}
|\scr{F}_\t\alpha_x(\nu)|\le (1+|\phi_x|^2)^{m'/2}p(|\nu|) 
\end{equation}
where $p\colon [0,\infty) \rightarrow [0,\infty)$ is a rapidly decreasing monotone function that can be chosen independently of $x \in X'$. Equations \eqref{e:pairingwithf}, \eqref{e:dombyrapid} show that the pairing $\pair{\Ch^u(\sf{E},\cdot)\A\S^u(\sigma,\cdot)}{f} \in \Omega(X^u)$ is integrable, and hence the right hand side of \eqref{e:deloc0} is well-defined.

To prove the result, we carry out non-abelian localization for the right hand side of \eqref{e:deloc0} and verify that the result coincides with \eqref{e:AtiyahSinger}. (The result is not an immediate consequence of \cite{Paradan98}, because the latter treats equivariant forms with \emph{polynomial} coefficients, and the results typically do not extend to smooth coefficients needed here.) Let  $\d_\xi=\d+2\pi \i\iota(\xi_X)$ denote the equivariant differential. The localization works more generally, for a $\d_\xi$-closed equivariant differential form $\alpha(\xi)$ with $C^\infty(\t)$-coefficients such that (i) the support of $\alpha(\xi)$ is contained in a $\Pi$-invariant subset $X' \subset X^\circ$ such that $X'/\Pi \subset X^\circ/\Pi$ is compact (in particular, $\phi$ is proper on the support of $\alpha(\xi)$), and (ii) $\scr{F}_\t\alpha_x$ satisfies an estimate as in \eqref{e:dombyrapid} where $p$ is rapidly decreasing and independent of $x\in X'$.

Recall Paradan's closed equivariant form with generalized coefficients, which appears in the index formula \eqref{e:cohomnonabel}:
\[ P_\beta(\xi)=\chi_\beta-\d\chi_\beta \cdot \Theta \int_0^\infty e^{s\d_\xi\Theta}\d s, \qquad \Theta=g(\kappa)\in \Omega^1(X)^T, \]
where $\chi_\beta \colon \ol{X}\rightarrow [0,1]$ is a smooth $T$-invariant bump function with compact support near $\Z_\beta$, equal to $1$ on a neighborhood of $\Z_\beta$, and such that $\{\chi_\beta|\beta \in \B\}$ have pairwise disjoint supports. Let
\[ \chi=\sum_{\beta \in \B} \chi_\beta, \qquad P(\xi)=\sum_{\beta \in \B}P_\beta(\xi)=\chi-\d\chi \cdot \Theta \int_0^\infty e^{s\d_\xi\Theta}\d s.\]
Thanks to Propositions \ref{p:weakperiodicity} and \ref{p:Kirwanpos}, the $\chi_\beta$ can be chosen such that: (i) the $\chi_\beta$ are all translates by elements of $\Pi$ of a finite collection $\chi_{\beta_*}$, $\beta^* \in \B^*$ (in particular $\d \chi$ is bounded), (ii) there is a constant $c>0$ such that $|\kappa(x)|\ge c$ for all $x \in \tn{supp}(1-\chi)$. Property (ii) plays a particularly important role below.

The form $P(\xi)$ satisfies (\cite[Proposition 2.1]{Paradan98})
\[1=P(\xi)-\d_\xi \delta(\xi), \qquad \delta(\xi)=(1-\chi)\Theta\int_0^\infty e^{s\d_\xi\Theta}\d s, \]
where as above, the meaning of the integral in the definition of $\delta(
\xi)$ is that the integrand should be paired with a test function on $\t$, before the integral with respect to $s$ is carried out. The desired result amounts to proving that
\begin{equation} 
\label{e:intexact}
\int_{X^u}\int_{\t}\d \xi \; \alpha(\xi)e^{\varpi+2\pi \i\phi\cdot \xi}\d_\xi \delta(\xi)=0.
\end{equation}
Let $r_j>0$ be a sequence of regular values of $|\phi|\colon X \rightarrow [0,\infty)$ such that $r_j \rightarrow \infty$ as $j \rightarrow \infty$, and let $X_j=|\phi|^{-1}([0,r_j])\cap X^\circ$, a smooth manifold with boundary. By Stokes' theorem, and using the definition of $\delta(\xi)$, \eqref{e:intexact} is equivalent to
\begin{equation}
\label{e:intexact2}
\lim_{j\rightarrow \infty}\int_{\partial X_j^u}(1-\chi)\Theta e^{\varpi} \int_0^\infty \d s \; e^{s\d \Theta}\int_{\t}\d\xi \; \alpha(\xi)e^{2\pi \i(\phi+s\phi_\Theta)\cdot \xi}=0,
\end{equation} 
where $\phi_\Theta\cdot \xi:=\iota(\xi_X)\Theta=g(\kappa,\xi_X)$. Performing the integral over $\t$ on the left hand side of \eqref{e:intexact2} yields
\begin{equation} 
\label{e:intexact3}
\lim_{j\rightarrow \infty}\int_{\partial X_j^u}(1-\chi)\Theta e^{\varpi} \int_0^\infty \d s \; e^{s\d \Theta}(\scr{F}_\t\alpha)(-\phi-s\phi_\Theta). 
\end{equation}
On $\tn{supp}(1-\chi) \cap \partial X_j^u$, we have
\begin{equation} 
\label{e:boundkappa}
|\phi+s\phi_\Theta|^2=|\phi|^2+s^2|\phi_\Theta|^2+2s|\kappa|^2\ge r_j^2+2c^2 s,
\end{equation}
with $c>0$. For $j$ large there are estimates, valid on $\partial X^u_j$, of the form
\begin{equation} 
\label{e:polybounds}
\tn{vol}(\partial X_j^u)\le C|r_j|^{\dim(\t)-1}, \quad |\Theta|\le C|r_j|, \quad |e^{s\d \Theta}|\le C(1+s)^{n/2}|r_j|^{n/2}, \quad |e^{\varpi}|\le C. 
\end{equation}
Therefore for $j$ large the expression \eqref{e:intexact3} is bounded by a constant times
\[ \lim_{j\rightarrow \infty} |r_j|^{\dim(\t)+n/2+m'/2}\int_0^\infty \d s\; (1+s)^{n/2}p\Big(\sqrt{r_j^2+2c^2 s}\,\Big), \]
which vanishes in the limit as $j \rightarrow \infty$ because $p$ is rapidly decreasing.
\end{proof}

\subsection{Formal series, formal change of variables and Poisson summation}\label{s:formal}
This section discusses preliminaries on formal deformations needed for the treatment of the fixed-point formula in the next section. Throughout $t$ is a formal variable, and if $A$ is an abelian group then $\bC A\ft$ denotes the space of formal power series in the variable $t$ with coefficients in $\bC \otimes A$.

The collection of $\infty$-jets of curves in the complexified torus $T_\bC$ forms a group $\J^\infty T_\bC$. In algebro-geometric terminology, these are the $\bC\ft$-points of $T_\bC$. An $\infty$-jet with base point $g \in T_\bC$ will be denoted $g_t \in \J^\infty_g T_\bC$ below, where the parameter $t \in \bR$ or $\bC$ (in the latter case we think of $\J^\infty T_\bC$ as holomorphic jets). 

An element $g_t \in \J^\infty_g T_\bC$ has a unique factorization
\begin{equation} 
\label{e:gz}
g_t=g\exp(g_t^+), \qquad g_t^+=\log(g^{-1}g_t) \in t\t_{\bC}\ft. 
\end{equation}
Jets of maps $T \rightarrow \bC$ based at $g \in T$ are in 1-1 correspondence with holomorphic jets of maps $T_\bC \rightarrow \bC$ based at $g$. If $F \in C^\infty(T)$ and $g \in T$, let $\J^\infty_{hol,g} F$ denote the holomorphic jet corresponding to the Taylor series of $F$ at the point $g$ (put more simply, we treat the variables in the Taylor series of $F$ at $g$ as complex variables). For $g_t \in \J^\infty_g T_\bC$ we define $F(g_t)=(\J^\infty_{hol,g} F) \circ g_t \in \J^\infty \bC=\bC\ft$ to be the composition of jets; in other words $F(g_t)$ is the result of substituting $g_t^+\in t\t_\bC\ft$ into the Taylor series of $F$ at the point $g$ and expanding in powers of $t$. The map 
\[ f \in C^\infty(T)\mapsto f(g_t)\in \bC\ft \] 
defines a formal series of distributions on $T$ supported at $g\in T$; we denote this formal series by $\delta_{g_t}\in \D'(T)\ft$. More generally if $f_t \in C^\infty(T)\ft$, say $f_t=f_0+tf_1+\cdots$, then $f_t(g_t)=f_0(g_t)+tf_1(g_t)+\cdots \in \bC\ft$. 

In the next section we encounter a deformation described by a formal 1-parameter family of vector fields 
\[ v_t=\sum_{j\ge 0}t^j v_j \in \t_\bC \otimes \bC R(T)\ft, \qquad v_j \in \t_\bC\otimes \bC R(T) \] 
on the complexified torus $T_\bC$. The latter determines an $\infty$-jet $\Phi_t$ of diffeomorphisms deforming the identity
\begin{equation}
\label{e:Phiz}
\Phi_t(u)=u\exp(t\ell^{-1}v_t(u)),
\end{equation}
where we have included the re-scaling $\ell^{-1}$ for later convenience. The corresponding $\Pi$-periodic $\infty$-jet of maps $\t_\bC\rightarrow \t_\bC$ is $\xi\mapsto \xi+t\ell^{-1}v_t(\exp(\xi))$. If $g_t \in \J^\infty_gT_\bC$ and $\Phi_t$ is as in \eqref{e:Phiz} then $\Phi_t(g_t)\in \J^\infty_g T_\bC$ is defined by composition of jets. Let $\d v_t \in \End(\t_\bC)\otimes\bC R(T)\ft$ be the differential (in the $T_\bC$ variables), and let $\det(1+t\ell^{-1}\d v_t)\in 1+t\bC R(T)\ft$ denote the complex determinant. The fixed-point problem
\[ \Phi_t(g_t)=g, \qquad g_t \in \J_g^\infty T_\bC \] 
is equivalent to
\begin{equation} 
\label{e:lambdaz}
g_t^++t\ell^{-1}v_t(g\exp(g_t^+))=0.
\end{equation}
for the formal series $g_t^+ \in t\t_\bC \ft$. On expanding each $v_j$ in Taylor series at $g$,
\[ v_j(g\exp(\xi))=\sum_{\alpha\ge 0}\frac{\xi^\alpha}{\alpha !}\partial^\alpha v_j(g),\]
equation \eqref{e:lambdaz} becomes a sequence of algebraic equations for the coefficients in the expansion of $g_t^+$ in powers of $t$ that can be solved recursively. For example the first two terms are
\[ g_t^+=-t\ell^{-1}v_0(g)-t^2(\ell^{-1}v_1(g)-\ell^{-2}\partial_{v_0(g)} v_0(g))+\cdots.\]
The equations for the higher terms quickly become more complicated, but in any case \eqref{e:lambdaz} has a unique solution. Define $\psi_t\in \t_\bC \otimes \bC R(T)\ft$ by $t\psi_t(g)=g_t^+$ where $g_t^+$ is the unique solution of \eqref{e:lambdaz}. Let $\Psi_t(u)=u\exp(t\psi_t(u))$, an $\infty$-jet of diffeomorphisms $T_\bC \rightarrow T_\bC$ deforming the identity. By construction, the composition
\[ \Phi_t\circ \Psi_t=\tn{Id}. \]
Recall the finite subgroup $T_\ell=\ell^{-1}\Lambda/\Pi \subset \t/\Pi=T$.
\begin{theorem}
\label{t:qPoisson}
Let $\Phi_t(u)=u\exp(t\ell^{-1}v_t(u))$, $v_t \in \t_\bC\otimes \bC R(T)\ft$. Then
\begin{equation} 
\label{e:qPoisson}
\sum_{\eta \in \Pi} \Phi_t(u)^{\ell\eta}=\frac{1}{|T_\ell|}\sum_{g \in T_\ell} \frac{\delta_{g_t}(u)}{\det(1+t\ell^{-1}\d v_t(g_t))} \in \D'(T)\ft
\end{equation}
where $g_t \in \J^\infty_g T_\bC$ is the unique solution to the fixed-point equation $\Phi_t(g_t)=g$.
\end{theorem}
\begin{proof}
By expanding the left hand side of \eqref{e:qPoisson} in powers of $t$ and using the Poisson summation formula, one checks that it is a well-defined formal series of distributions with support contained in $T_\ell$. It suffices to check both sides of \eqref{e:qPoisson} agree when paired with test functions $f \in \bC R(T)$. The function $f$ extends uniquely to a holomorphic function on $T_\bC$ that we also denote by $f$. Using Borel summation (i.e. replacing $t^j$ by $t^j \rho_j(t)$, where $\rho_j$ is a bump function on $\bR$ with $\rho_j=1$ near $t=0$, and supports shrinking to $0$ as $j \rightarrow \infty$ at a rate depending on the properties of the series $\psi_t$) one constructs a smooth function $\bm{\psi}\colon \bR \times T_\bC \rightarrow \t_\bC$, denoted $(t,u)\in \bR \times T_\bC \mapsto \bm{\psi}_t(u)$ (holomorphic in $u$), whose Taylor series at $t=0$ is $\psi_t$. Let $\bm{\Psi}_t(u)=u\exp(t\bm{\psi}_t(u))$ be the corresponding map $\bR \times T_\bC \rightarrow T_\bC$. Since $\bm{\Psi}_0$ is the identity, smoothness in $t$ implies that, for $t$ small, $\bm{\Psi}_t$ restricts to a biholomorphism from a relatively compact annulus $\scr{A}\subset T_\bC$ containing $T$ to its image $\bm{\Psi}_t(\scr{A})$. Let $\bm{\Phi}_t$ denote its inverse. Note that the Taylor series of $\bm{\Phi}_t$ at $t=0$ is $\Phi_t$, and consequently the Taylor series of $\det(\d\bm{\Phi}_t)$ at $t=0$ is $\det(\d\Phi_t)=\det(1+t\ell^{-1}\d v_t)$. We will prove that for $t$ sufficiently small,
\begin{equation} 
\label{e:qPoisson2onT}
\sum_{\eta \in \Pi} \int_T f(u)\bm{\Phi}_t(u)^{\ell\eta}\d u=\frac{1}{|T_\ell|}\sum_{g \in T_\ell}\frac{f(\bm{\Psi}_t(g))}{\det(\d \bm{\Phi}_t(\bm{\Psi}_t(g)))}
\end{equation}
from which \eqref{e:qPoisson} follows by taking Taylor series in $t$ of both sides at $t=0$. 

For $u \in T$ define
\[ F(t,u)=\frac{1}{|T_\ell|}\sum_{g \in T_\ell}\frac{f(\bm{\Psi}_t(gu))}{\det(\d \bm{\Phi}_t(\bm{\Psi}_t(gu)))},\]
hence the right hand side of \eqref{e:qPoisson2onT} is $F(t,1)$. By construction $\xi\mapsto F(t,\exp(\xi))$ is not only $\Pi$-periodic but $\ell^{-1}\Lambda$-periodic, hence has a Fourier expansion (over the dual lattice $(\ell^{-1}\Lambda)^*=\ell \Pi$),
\begin{equation} 
\label{e:Fourierexp}
F(t,h)=\sum_{\eta \in \Pi} \hat{F}(t,\eta)h^{-\ell \eta} 
\end{equation}
where
\begin{align*}
\hat{F}(t,\eta)&=\int_{\t/\ell^{-1}\Lambda} F(t,u)u^{\ell \eta}\d u\\
&=\int_T \frac{f(\bm{\Psi}_t(u))}{\det(\d \bm{\Phi}_t(\bm{\Psi}_t(u)))}u^{\ell \eta}\d u.
\end{align*}
For the second equality, the sum over $T_\ell$ and integral over $\t/\ell^{-1}\Lambda$ were combined to yield the integral over $T=\t/\Pi$ (note $g^{\ell \eta}=1$ for $g \in T_\ell=\ell^{-1}\Lambda/\Pi$ and $\eta \in \Pi$). This is a contour integral over the contour $T \subset T_\bC$. Making the change of variables $\bm{\Psi}_t(u)\leadsto u$, the Jacobian factor is absorbed,\footnote{When $\dim(\t)=1$ for example, the change of variables here is of the kind
\[ \oint_\C g(\bm{\Psi}(w))\bm{\Psi}'(w)\d w=\oint_{\bm{\Psi}(\C)} g(w)\d w.\]}
and the integral becomes
\[ \hat{F}(t,\eta)=\int_{\bm{\Psi}_t(T)} f(u)\bm{\Phi}_t(u)^{\ell\eta}\]
with $\bm{\Psi}_t(T)\subset T_\bC$ the new contour. The integrand is a holomorphic function of $h$, so the integral does not change if the contour is deformed back to $T$ through the family of contours $\bm{\Psi}_{st}(T)$, $s \in [0,1]$, hence
\[ \hat{F}(t,\eta)=\int_T f(u)\bm{\Phi}_t(u)^{\ell\eta}\]
Substituting in \eqref{e:Fourierexp} and setting $h=1$ yields the result.
\end{proof}

\subsection{Fixed point formula}\label{s:abelloccoh}
In this section we use Theorem \ref{t:deloc} to derive a fixed-point formula for the index, under an additional `twisted' $\Pi$-equivariance assumption on the equivariant Chern character form.

Let $\sf{E}_t \in \bC K^\ad_T(\M)\ft$ be a formal series of admissible classes with $\bC$ coefficients. By applying the index map term-by-term in $t$, the index
\[ \index_T(\sf{E}_t)\in \bC R^{-\infty}(T)\ft \]
is defined. For example, if $\sf{E}\in K^\ad_T(\M)$ is admissible, then $\sf{E}_t=\exp(t\sf{E}) \in \bC K^\ad_T(\M)\ft$ is such a formal series, and
\[ \index_T(\exp(t\sf{E}))=\sum_{n=0}^\infty \frac{t^n}{n!}\index_T(\sf{E}^{\otimes n})\in \bC R^{-\infty}(T)\ft. \]
Without any additional difficulty one could consider series involving multiple formal variables.

\begin{definition}
\label{d:twistedequivariance}
Let $\sf{E}_t \in \bC K^\ad_T(\M)\ft$ be a series of admissible classes. We say that $\sf{E}_t$ \emph{admits infinitesimally twisted $T\times \Pi$-equivariant Chern character forms} if there exists a formal series of vector fields $v_t \in \t_\bC\otimes \bC R(T)\ft$ and a formal series of Chern character forms $\Ch^u(\sf{E}_t,\xi) \in \Omega(X^u)$ as in Definition \ref{d:uniformChern}, such that for all $\eta \in \Pi$,
\begin{equation} 
\label{e:etachern}
\eta^*\Ch^u(\sf{E}_t,\xi)=e^{2\pi \i tv_t(u\exp(\xi))\cdot \eta}\Ch^u(\sf{E}_t,\xi).
\end{equation}
We refer to the formal series $v_t$ as the \emph{infinitesimal twist}.
\end{definition}
This definition is motivated by Example \ref{ex:dbarondisk} and similar examples for more general $\Sigma$, taken up in detail in the sequel to this article. In brief, in these examples one has an admissible class $\sf{E}$ satisfying $\eta^*\sf{E}=\sf{E}+\nu\cdot \eta$ in $K^\ad_T(\M)$ for some $\nu \in \t\otimes R(T)$, and thus $\sf{E}_t=\exp(t\sf{E})$ obeys an equation of the form \eqref{e:etachern} in K-theory.
\begin{corollary}
\label{c:twistedinvariance}
Let $\sf{E}_t \in \bC K^\ad_T(\M)\ft$ admit infinitesimally twisted $T\times \Pi$-equivariant Chern character forms $\Ch^u(\sf{E}_t,\xi)$. Let $\Phi_t(u)=u\exp(t\ell^{-1}v_t(u))$ where $v_t$ is the infinitesimal twist. Then $\index_T(\sf{E}_t)\in \bC R^{-\infty}(T)\ft$ satisfies
\[ \Phi_t(g)^{\ell \eta}\index_T(\sf{E}_t)(g)=\index_T(\sf{E}_t)(g), \qquad \forall \eta \in \Pi.\]
\end{corollary}
\begin{proof}
By Theorem \ref{t:deloc}, for $\xi$ sufficiently small
\[ \index_T(\sf{E}_t)(u\exp(\xi))=\int_{X^u} \Ch^u(\sf{E}_t,\xi)\A\S^u(\sigma,\xi).\]
Let $g=u\exp(\xi)$. Multiplying both sides by $\Phi_t(g)^{\ell \eta}$, using Definition \ref{d:twistedequivariance} as well as equation \eqref{e:ASPiEquiv}, yields
\[ \Phi_t(g)^{\ell \eta}\index_T(\sf{E}_t)(g)=\int_{X^u}\eta^* \big(\Ch^u(\sf{E}_t,\xi)\A\S^u(\sigma,\xi)\big)=\index_T(\sf{E}_t)(g)\]
since $X^u$ is $\Pi$-invariant.
\end{proof}

For a $T$-equivariant form $\alpha^g(\xi) \in \Omega(X^g)$ depending smoothly on $\xi$ (such as the Chern character form or Atiyah-Singer form considered above), we define the formal series of differential forms
\[ \alpha^{g_t}=\alpha^g(g_t^+)\in \Omega(X^g)\ft \]
where, similar to our discussion of functions in Section \ref{s:formal}, the right hand side is the result of substituting $g_t^+\in t\t_\bC\ft$ into the Taylor series at $0$ of the function $\xi \in \t \mapsto \alpha^g(\xi)\in \Omega(X^g)$ and expanding in powers of $t$.

\begin{theorem}
\label{t:abelianloc}
Let $\sf{E}_t \in \bC K^\ad_T(\M)\ft$ admit infinitesimally twisted $T\times \Pi$-equivariant Chern character forms $\Ch^u(\sf{E}_t,\xi)$. Then $\index_{T}(\sf{E}_t)$ has support contained in the finite set $T_\ell=\ell^{-1}\Lambda/\Pi \subset T$. Let $\Phi_t(u)=u \exp(t\ell^{-1}v_t(u))$ where $v_t$ is the infinitesimal twist. Let $g_t \in \J^\infty_g T_\bC$ be the solution of the fixed-point equation $\Phi_t(g_t)=g$. Then $\Ch^{g_t}(\sf{E}_t)\A\S^{g_t}(\sigma)\in \Omega(X^g)\ft$ is $\Pi$-invariant, and
\[ \index_{T}(\sf{E}_t)(h)=\frac{1}{|T_\ell|}\sum_{g \in T_\ell}\frac{\delta_{g_t}(h)}{\det(1+t\ell^{-1}\d v_t(g_t))}\int_{X^g/\Pi} \Ch^{g_t}(\sf{E}_t)\A\S^{g_t}(\sigma). \]
\end{theorem}
\begin{proof}
For $u \in T$ fixed and $\xi \in \t$ sufficiently small, let
\[ \alpha^u(\xi)=\Ch^u(\sf{E}_t,\xi)\A\S^u(\sigma,\xi) \]
denote the integrand in Theorem \ref{t:deloc}, and let $h=u\exp(\xi)$. Let $X_\diamond\subset X$ denote a fundamental domain for the $\Pi$ action. By Definition \ref{d:twistedequivariance} and equation \eqref{e:ASPiEquiv},
\[ \index_T(\sf{E}_t)(h)=\sum_{\eta \in \Pi}\Phi_t(h)^{\ell\eta}\int_{X_\diamond^u} \alpha^u(\xi).\]
The sum over $\Pi$ can be evaluated using Theorem \ref{t:qPoisson}:
\begin{equation} 
\label{e:sumoverPi}
\index_T(\sf{E}_t)(h)=\frac{1}{|T_\ell|}\sum_{g \in T_\ell}\frac{\delta_{g_t}(h)}{\det(1+t\ell^{-1}\d v_t(g_t))}\int_{X_\diamond^u} \alpha^u(\xi),
\end{equation}
where $g_t \in \J^\infty_g T_\bC$ is the solution of the fixed-point equation $\Phi_t(g_t)=g$. 

Fix $g \in T_\ell$ and consider the corresponding term in \eqref{e:sumoverPi}. Recall $h=u\exp(\xi)$. Since $u$ is fixed, while $\xi$ is small and variable, the Dirac delta distribution $\delta_{g_t}(h)$ leads to a constraint on $\xi$. Solving the constraint yields a series
\[ \xi_t=\xi_0+t\xi_1+\cdots=\log(u^{-1}g)+g_t^+ \in \t_\bC \ft \]
where $g_t=g\exp(g_t^+)$, $g_t^+\in t\cdot \t_\bC \ft$, and $\log(-)$ denotes the preferred logarithm when the argument is near $1 \in T_\bC$. The Dirac delta distribution in \eqref{e:sumoverPi} means we may replace $\alpha^u(\xi)$ by $\alpha^u(\xi_t)$. We claim that $\alpha^u(\xi_t)$ is $\Pi$-invariant, hence descends to $X^u/\Pi$. Indeed using Definition \ref{d:twistedequivariance} and equation \eqref{e:ASPiEquiv} once again, we have
\[ \eta^*\alpha^u(\xi_t)=\Phi_t(g_t)^{\ell\eta}\alpha^u(\xi_t)=g^{\ell \eta}\alpha^u(\xi_t)=\alpha^u(\xi_t), \]
using the definition of $g_t$ and the equation $g^{\ell \eta}=1$ which holds because $g \in T_\ell=\ell^{-1}\Lambda/\Pi$. Thus \eqref{e:sumoverPi} becomes
\[ \index_T(\sf{E}_t)(h)=\frac{1}{|T_\ell|}\sum_{g \in T_\ell}\frac{\delta_{g_t}(h)}{\det(1+t\ell^{-1}\d v_t(g_t))}\int_{X^u/\Pi} \alpha^u(\xi_t).\]
The integrand is closed for the differential $\d_{\xi_t}=\d+2\pi \i \iota(\xi_t)_M$ (where $(\xi_t)_M \in C^\infty(M,TM_\bC)\ft$ denotes the formal series of vector fields $(\xi_0)_M+t(\xi_1)_M+\cdots$) and localizes to the fixed point set of $\xi_0=\log(u^{-1}g)$ in $X^u$, which is $X^{u\exp(\xi_0)}=X^g$ (this holds when $\xi_0$ is sufficiently small, as we are assuming, cf. \cite{vergne1996multiplicities}). A small addendum to the usual abelian localization formula (see Lemma \ref{l:qabelloc} below)   yields
\[ \index_T(\sf{E}_t)(h)=\frac{1}{|T_\ell|}\sum_{g \in T_\ell}\frac{\delta_{g_t}(h)}{\det(1+t\ell^{-1}\d v_t(g_t))}\int_{X^g/\Pi} \frac{\iota^*\alpha^u(\xi_t)}{\Eul(\nu,\xi_t)},\]
where $\nu$ is the normal bundle to $X^g$ in $X^u$ and $\iota \colon X^g/\Pi \hookrightarrow X^u/\Pi$ is the inclusion. Using the properties of the Atiyah-Segal-Singer integrand and the bouquet property of the equivariant Chern forms, one finds (cf. \cite{BerlineVergneKirillovFormula})
\[ \frac{\iota^*\alpha^u(\xi_t)}{\Eul(\nu,\xi_t)}=\alpha^g(g_t^+)=\alpha^{g_t}.\]
This completes the proof.
\end{proof}

In the proof of the result above we used the following addendum to the standard abelian localization formula.
\begin{lemma}
\label{l:qabelloc}
Let $M$ be a $T$-manifold where $T$ is a compact torus with Lie algebra $\t$. Let $\xi_t=\sum_{j\ge 0}t^j\xi_j \in \t_\bC\ft$. Let $\beta \in \Omega_c(M)\ft^T$ be a compactly-supported $\d_{\xi_t}$-closed $T$-invariant $\bC$-valued differential form. Then
\[ \int_M \beta=\int_{M^{\xi_0}}\frac{\beta}{\Eul(\nu,\xi_t)}.\]
\end{lemma}
\begin{proof}
We start by reviewing the standard Atiyah-Bott-Berline-Vergne abelian localization formula \cite{AtiyahBottLocalization,BerlineVergne1,BerlineVergne2}. For $\xi \in \t_\bC$ define the differential
\[\d_\xi=\d+2\pi \i \iota(\xi_M). \]
It squares to zero on the subspace $\Omega(M)^T$ of $T$-invariant forms. One has
\begin{equation} 
\label{e:intzero}
\int_M \d_\xi \beta=0 
\end{equation}
for any $\beta \in \Omega_c(M)$ (the subscript `c' meaning compactly supported), by Stokes' theorem and because $\iota(\xi_M)\beta$ has exterior degree less than $\dim(M)$.  

Let $g$ be a $T$-invariant Riemannian metric, which we extend $\bC$-bilinearly to $TM_\bC$. For $\xi \in \t_\bC$, let $\bar{\xi}$ be the complex conjugate, and define $\theta_\xi \in \Omega^1(M)$ by
\begin{equation} 
\label{e:thetaxi}
\theta_\xi(\cdot)=g(\bar{\xi}_M,\cdot).
\end{equation}
Then
\[ \d_\xi\theta_\xi=\d \theta_\xi+2\pi \i g(\bar{\xi}_M,\xi_M).\]
Since the component of $\d_\xi \theta_\xi$ of exterior degree $0$ is non-zero away from $M^\xi=\{\xi_M=0\}$, $\d_\xi\theta_\xi$ is invertible away from this submanifold. 

Now let $\xi_0 \in \t_\bC$ and let $\chi$ be a $T$-invariant bump function equal to $1$ near $M^{\xi_0}$ and with support contained in a tubular neighborhood $U$ of the latter. For $\xi \in \t_\bC$ sufficiently close to $\xi_0$ one has $M^\xi\subset M^{\xi_0}$, and consequently $\chi$ is also equal to $1$ on a neighborhood of $M^\xi$. Thus for $\xi$ sufficiently close to $\xi_0$, the differential form
\begin{equation} 
\label{e:Pxi}
P(\xi)=\chi+\d\chi \frac{\theta_\xi}{\d_\xi \theta_\xi} \in \Omega(M),
\end{equation}
is well-defined, as $\d\chi$ vanishes on a neighborhood of the subset $M^\xi$ where $\d_\xi \theta_\xi$ fails to be invertible.  By construction $P(\xi)\in \Omega(M)$ has support contained in a tubular neighborhood of $M^{\xi_0}$. Note that $(1-\chi)$ vanishes near $M^\xi$, so $(1-\chi)(\d_\xi\theta_\xi)^{-1}$ is also well-defined, which implies that the difference
\begin{equation} 
\label{e:PxiOne}
1-P(\xi)=\d_\xi\Big((1-\chi)\frac{\theta_\xi}{\d_\xi \theta_\xi}\Big)
\end{equation}
is $\d_\xi$-exact. Consequently if $\beta \in \Omega_c(M)^T$ is $\d_\xi$-closed then the integral
\[ \int_M \beta = \int_M \beta\cdot P(\xi)=\int_U \beta\cdot P(\xi) \]
localizes near $M^{\xi_0}$. Using the Thom isomorphism and the inverse of the equivariant Euler form $\Eul(\nu,\xi)$ of the normal bundle $\nu=\nu(M,M^{\xi_0})$, one easily reduces further to an integral over $M^{\xi_0}$, giving
\begin{equation} 
\label{e:abelloc}
\int_M \beta=\int_{M^{\xi_0}} \frac{\beta}{\Eul(\nu,\xi)}.
\end{equation}
This is the usual abelian localization formula.

Since \eqref{e:abelloc} depends smoothly on $\xi$ sufficiently close to $\xi_0$, it easily adapts to `formal deformations' $\xi_t=\xi_0+t\xi_1+t^2\xi_2+\cdots$ where $t$ is a formal variable. Just as above, $\d_{\xi_t}$ restricts to a differential on $\Omega(M)\ft^T$, and has the property in \eqref{e:intzero} for any $\beta \in \Omega_c(M)\ft$. We define $\theta_{\xi_t} \in \Omega^1(M)\ft^T$ by substituting $\xi_t$ for $\xi$ in \eqref{e:thetaxi}. Similar to before, the $t^0\cdot \Omega^0(M)$-component of $\d_{\xi_t}\theta_{\xi_z}$ is invertible away from $M^{\xi_0}$, hence we may define $P(\xi_t)$ by substituting $\xi_t$ for $\xi$ in \eqref{e:Pxi}, and the obvious analogue of \eqref{e:PxiOne} holds. The remaining discussion also goes through, and proves the lemma.
\end{proof}

\subsection{Weyl symmetry}
For $\sf{E} \in K^\ad_{N(T)}(X)$, we require that an equivariant Chern character $\Ch^u(\sf{E},\xi) \in \Omega(X^u)$ satisfy the bouquet condition for $N(T)$ in addition to the conditions in Definition \ref{d:uniformChern}. In particular $n_w\cdot\Ch^u(\sf{E},\xi)=\Ch^{wu}(\sf{E},w\xi)$ in $\Omega(X^{wu})$ where $n_w \in N(T)$ has image $w \in W=N(T)/T$. In the case of a series $\sf{E}_t\in \bC K^\ad_{N(T)}(X)\ft$ we require the infinitesimal twist $v_t \in (\t_\bC\otimes \bC R(T)\ft)^W$ be $W$-invariant. If these two additional conditions are satisfied then we say that $\sf{E}_t$ \emph{admits infinitesimally twisted $N(T)\ltimes \Pi$-equivariant Chern character forms}. In this case the $W$-antisymmetry (Proposition \ref{p:Weylanti}) and the twisted $\Pi$-symmetry (Corollary \ref{c:twistedinvariance}) combine into a twisted $W_{\tn{aff}}$-antisymmetry.

Recall the finite subgroup $T_\ell=\ell^{-1}\Lambda/\Pi\subset T$. Let $T_\ell^\reg=T_\ell\cap T^\reg$, where $T^\reg$ is the set of regular elements.
\begin{theorem}
\label{t:regsupp}
Let $\sf{E}_t \in \bC K^\ad_G(\M)\ft$ admit infinitesimally twisted $N(T)\ltimes \Pi$-equivariant Chern character forms $\Ch^u(\sf{E}_t,\xi)$. Let $\Phi_t(u)=u\exp(t\ell^{-1}v_t(u))$ where $v_t \in (\t_\bC\otimes \bC R(T)\ft)^W$ is the infinitesimal twist. Let $g_t \in \J^\infty_g T_\bC$ be the solution of the fixed-point equation $\Phi_t(g_t)=g$. Then $\index_T(\sf{E}_t)\in \bC R^{-\infty}(T)\ft$ has support contained in $T_\ell^\reg$ and
\begin{equation} 
\label{e:regsupp}
\index_T(\sf{E}_t)(h)=\frac{1}{|T_\ell|}\sum_{g \in T_\ell^{\reg}/W}\frac{\sum_{w \in W}(-1)^{l(w)}\delta_{wg_t}(h)}{\det(1+t\ell^{-1}\d v_t(g_t))}\int_{X^g/\Pi} \Ch^{g_t}(\sf{E}_t)\A\S^{g_t}(\sigma). 
\end{equation}
\end{theorem}
\begin{proof}
For $g \in T_\ell$, let
\[ f(g_t)=\frac{1}{\det(1+t\ell^{-1}\d v_t(g_t))}\int_{X^g/\Pi} \Ch^{g_t}(\sf{E}_t)\A\S^{g_t}(\sigma),\]
so that by Theorem \ref{t:abelianloc},
\[ \index_T(\sf{E}_t)(h)=\frac{1}{|T_\ell|}\sum_{g \in T_\ell}\delta_{g_t}(h)f(g_t).\]
Let $w \in W$ and let $n_w \in N(T)$ be a lift. By the bouquet property
\[ \Ch^{wg}(\sf{E}_t,w\xi)=n_w\cdot \Ch^g(\sf{E}_t,\xi)\]
hence $\Ch^{wg_t}(\sf{E}_t)=n_w\cdot \Ch^{g_t}(\sf{E}_t)$. On the other hand, because of the $W$-antisymmetry of the Thom-Bott element $\scr{B}$, the Atiyah-Singer integrand obeys
\[ \A\S^{wu}(\sigma,w\xi)=(-1)^{l(w)}n_w\cdot \A\S^u(\sigma,\xi) \]
in cohomology, hence $\A\S^{wg_t}(\sigma)=(-1)^{l(w)}n_w\cdot \A\S^{g_t}(\sigma)$. Consequently
\begin{equation} 
\label{e:symmf}
f(wg_t)=(-1)^{l(w)}f(g_t).
\end{equation}
Suppose $g \in T_\ell$ is non-regular. Then there is an element $w \in W$ of order $2$ such that $wg=g$. Since $v_t$ is $W$-invariant, the element $g_t \in \J^\infty_g(T_\bC)$ satisfies $wg_t=g_t$. Therefore \eqref{e:symmf} becomes
\[ f(g_t)=f(wg_t)=-f(g_t) \quad \Rightarrow \quad f(g_t)=0.\]
This shows that the support consists of regular elements. Choosing any set of representatives for the elements of $T_\ell^\reg/W$, the expression \eqref{e:regsupp} follows from \eqref{e:symmf}.
\end{proof}
Equation \eqref{e:regsupp} may be reformulated as an equality of $G$-invariant distributions on $G$:
\begin{equation} 
\label{e:inducedGdist}
\index_G(\sf{E}_t)=\frac{1}{|T_\ell|}\sum_{g \in T_\ell^{\reg}/W}\frac{(-1)^{|\mf{R}_+|}J(g_t)\delta_{g_t}}{\det(1+t\ell^{-1}\d v_t(g_t))}\int_{X^g/\Pi} \Ch^{g_t}(\sf{E}_t)\A\S^{g_t}(\sigma),
\end{equation}
where $\delta_{g_t} \colon C^\infty(G)^G\rightarrow \bC\ft$ is evaluation on the conjugacy class of $g_t$.

\bibliographystyle{amsplain}
\bibliography{../Biblio}

\end{document}